\documentclass{article}
\usepackage[journal=DM, lang=american]{ems-journal}

\theoremstyle{plain}
\newtheorem{theorem}{Theorem} 
\newtheorem{definition}{Definition}
\newtheorem{prop}{Proposition} 
\newtheorem{proposition}[definition]{Proposition}
\newtheorem{fact}[definition]{Fact}
\newtheorem{lemma}[definition]{Lemma}

\newtheorem{property}[definition]{Property}
\newtheorem{assumption}[definition]{Assumption}

\newtheorem{remark}[definition]{Remark}

\theoremstyle{remark}

\newcommand{\myP}{\mathcal{P}}
\newcommand{\Id}{\mathrm{Id}}
\newcommand{\R}{\mathbb{R}}
\newcommand{\myB}{\mathscr{B}}
\newcommand{\myC}{\mathscr{C}}
\newcommand{\D}{\mathscr{D}}

\newcommand{\M}{\mathbf{H}^m}
\newcommand{\Op}{{\mathrm Op}}
\newcommand{\J}{\mathcal{J}}
\newcommand{\Z}{\mathbb{Z}}
\newcommand{\N}{\mathbb{N}}
\newcommand{\sR}{\text{sR}}

\newcommand{\bqn}{\begin{equation}}
\newcommand{\eqn}{\end{equation}}

\numberwithin{equation}{section}

\begin{document}

\title{Quantum limits of sub-Laplacians via joint spectral calculus}

\emsauthor{1}{Cyril Letrouit}{C.~Letrouit}


\emsaffil{1}{Department of Mathematics, Massachussets Institute of Technology, Cambridge MA 02139, USA; and Département de mathématiques et applications, \'Ecole normale supérieure and CNRS, Université Paris Sciences et Lettres, 45 rue d'Ulm, 75005 Paris, France \email{letrouit@mit.edu}}

\classification[35S05, 35P05, 81Q10]{35H10}


\begin{abstract}
We establish two results concerning the Quantum Limits (QLs) of some sub-Laplacians. First, under a commutativity assumption on the vector fields involved in the definition of the sub-Laplacian, we prove that it is possible to split any QL into several pieces which can be studied separately, and which come from well-characterized parts of the associated sequence of eigenfunctions. 

Secondly, building upon this result, we study in detail the QLs of a particular family of sub-Laplacians defined on products of compact quotients of Heisenberg groups. We express the QLs through a disintegration of measure result which follows from a natural spectral decomposition of the sub-Laplacian in which harmonic oscillators appear.

Both results are based on the construction of an adequate elliptic operator commuting with the sub-Laplacian, and on the associated joint spectral calculus. They illustrate the fact that, because of the possible high degeneracies in the spectrum, the spectral theory of sub-Laplacians is very rich.
\end{abstract}

\maketitle

\section{Introduction and main results} \label{s:intro}

\subsection{Motivation}

The main goal of this paper is to establish some properties of the eigenfunctions of families of hypoelliptic operators in the high-frequency limit.  A typical problem is the description of the Quantum Limits (QL) of the operator, i.e., measures that are weak limits of a subsequence of squares of eigenfunctions. All the operators we consider in the sequel are sub-Laplacians, and they are in particular hypoelliptic as we will see shortly.

\subsubsection{Sub-Laplacians}

Let us now recall the general definition of a sub-Laplacian. Let $n\in\mathbb{N}^*$ and let $M$ be a smooth, connected and compact manifold of dimension $n$ without boundary. Let $X_1,\ldots,X_N$ be smooth vector fields on $M$ that are not necessarily independent but which satisfy the so-called H\"ormander bracketing condition 
\begin{equation}\label{e:hormander}
\begin{split}
&\text{The vector fields $X_1,\ldots,X_N$ and their iterated Lie brackets $[X_i,X_j],$}\\
 & \text{$[X_i,[X_j,X_k]]$, etc. span the tangent space $T_xM$ at every point $x\in M$.}
\end{split}
\end{equation}
Let also $\mu$ be  a smooth (positive) volume form on $M$. We consider the operator
\begin{equation} \label{e:defsubLapl}
\Delta_{\sR}=-\sum_{i=1}^N X_i^*X_i=\sum_{i=1}^N(X_i^2+\text{div}_\mu(X_i)X_i)
\end{equation}
where $X_i^*$ is the transpose of $X_i$ in $L^2(M,\mu)$. The ``sR'' index in $\Delta_{\sR}$ stands for ``sub-Riemannian''. The operator $\Delta_{\sR}$ is non-positive and self-adjoint. Then
\[ \mathcal{D}=\text{Span}(X_1,\ldots,X_N) \subset TM \]
is called the distribution. We note that Laplace--Beltrami operators are particular kinds of sub-Laplacians.\footnote{To see that the Laplace--Beltrami operator on a Riemannian manifold $(M,g)$ is a sub-Laplacian, take $\mu$ to be the Riemannian volume and take $X_1(q),\ldots,X_N(q)$ spanning $T_qM$ for any $q$ (without necessarily being independent, since this is not always possible globally), with lengths adjusted in a way that the principal symbol of $-\Delta_g$ coincides with the principal symbol of $\sum_{i=1}^N X_i^*X_i$. One can check that $\Delta_g=-\sum_{i=1}^N X_i^*X_i$.}

Under the assumption \eqref{e:hormander}, $\Delta_{\sR}$ is hypoelliptic (see \cite{hormander1967hypoelliptic}) and has a compact resolvent: this follows from subelliptic regularity estimates of the form
\[ \exists\, r\in\N, \exists\, C>0, \qquad \lVert u\rVert_{H^{1/r}(M)}^2\leq C\big(\lVert\Delta_{\sR}u\rVert^2_{L^2(M)}+\lVert u\rVert_{L^2(M)}^2\big) \]
(see \cite[Theorem 17 and estimate (17.20)]{rothschild76}). Thus there exists a sequence of (complex-valued) eigenfunctions of $-\Delta_{\sR}$ associated to the eigenvalues in increasing order $0=\lambda_1< \lambda_2\leq \cdots$ (with $\lambda_j\rightarrow +\infty$ as $j\rightarrow +\infty$) which is orthonormal for the $L^2(M,\mu)$ scalar product. 

\subsubsection{Quantum Limits} 
The main purpose of this paper is to understand the weak limits of the sequence of probability measures $|\varphi_k|^2d\mu$ where $(\varphi_k)_{k\in\N^*}$ is a sequence of normalized eigenfunctions of $-\Delta_{\sR}$ associated to eigenvalues tending to $+\infty$, for particular sub-Laplacians $\Delta_{\sR}$.

There is a phase space extension of these weak limits whose behaviour is also of interest. Let us recall the following definition (see~\cite{gerard1991microlocal}).

\begin{definition} \label{d:defmdm}
 Let $(u_k)_{k\in\N^*}$ be a bounded sequence in $L^2(M)$ and weakly converging to $0$. A microlocal defect measure of $(u_k)_{k\in\N^*}$ is any Radon measure $\nu$ on $S^*M$ for which there exists an extraction $\sigma:\N^*\rightarrow \N^*$ such that for any $0$-th order pseudodifferential operators $A$ with principal symbol $a=\sigma_P(A)$ (see Appendix \ref{a:pseudo}), there holds
 \begin{equation*}
 (Au_{\sigma(k)},u_{\sigma(k)})\underset{k\rightarrow +\infty}{\longrightarrow} \int_{S^*M}ad\nu.
 \end{equation*}
 Here, $(\cdot,\cdot)$ denotes the $L^2(M,\mu)$ scalar product.
 \end{definition}
Microlocal defect measures are useful tools for studying the (asymptotic) concentration and oscillation properties of sequences; note they are necessarily non-negative (see~\cite{gerard1991microlocal}).
\begin{definition}
Given a sequence $(\varphi_k)_{k\in\N^*}$ of eigenfunctions of $-\Delta_{\sR}$ with $\lVert\varphi_k\rVert_{L^2(M,\mu)}=1$, we call Quantum Limit (QL) any microlocal defect measure of $(\varphi_k)_{k\in\N^*}$.
\end{definition}
\begin{remark}
Since $\varphi_k$, $k\in\N^*$ is normalized, any QL is a probability measure on $S^*M$.
\end{remark}

For any Riemannian manifold $(M,g)$, it is well known that any QL $\nu$ of the Laplace--Beltrami operator $\Delta_g$ is invariant under the geodesic flow $\exp(t\vec{G})$. Here $G=\sigma_P(\sqrt{-\Delta_g})$ and $\vec{G}$ is the associated Hamiltonian vector field (the geodesic vector field), and the claim is that $\exp(t\vec{G})\nu=\nu$ for any $t\in \R$. To prove it, we note that for any sequence $(\varphi_k)_{k\in\N^*}$ consisting of normalized eigenfunctions of $-\Delta_g$, there holds
\begin{equation} \label{e:compRiem}
(A\sqrt{-\Delta_g}\varphi_k,\varphi_k)_{L^2}-(\sqrt{-\Delta_g}A\varphi_k,\varphi_k)_{L^2}=0
\end{equation}
for any $t\in\R$, any $k\in \mathbb{N^*}$ and any $0$-th order pseudodifferential operators $A$. It follows from the commutation rule for pseudodifferential operators that $\int_{S^*M}\{\sigma_P(A),G\}d\nu=0$, which in turn implies $\smash{\vec{G}\nu=0}$ and  $\smash{\exp(t\vec{G})\nu=\nu}$ for any $t\in\R$.

The structure and the invariance properties of the QLs of sub-Laplacians are more complicated than that of Riemannian Laplacians. To see it, let us consider a general sub-Laplacian $\Delta_{\sR}$, the principal symbol 
\[ g^*=\sigma_P(-\Delta_{\sR}), \] 
and the associated sub-Riemannian geodesic flow $\vec{g}^*$. The invariance of QLs of $\Delta_{\sR}$ under the sub-Riemannian geodesic flow $\vec{g}^*$ is still true, but it does not provide any information about the part of the QL lying in $(g^*)^{-1}(0)$ since the geodesic flow is stationary at such points.  Indeed, we note that the above computation \eqref{e:compRiem} does not work anymore for general sub-Laplacians since $\sqrt{-\Delta_{\sR}}$ is not a pseudodifferential operator near its characteristic cone $(g^*)^{-1}(0)$ (due to the blow-up of some derivatives of $\sqrt{g^*}$).

We denote by 
\[ \Sigma=(g^{*})^{-1}(0)=\mathcal{D}^{\perp}\subset T^*M \] 
the characteristic cone (where $\perp$ is in the sense of duality). We make the identification 
\begin{equation} \label{e:identification}
S^*M=U^*M\cup S\Sigma
\end{equation} 
where $S^*M$ is the cosphere bundle obtained by taking the quotient of the fibers of $T^*M$ by $\R^+$, $U^*M=\{g^*=1\}$ has any of its points identified with some point in $S^*M$ by homogeneity, and the remaining elements in $S^*M$ are the directions along which $g^*=0$. This last set can be identified with $S\Sigma$, the quotient of $\Sigma$ by $\R^+$.

In the sequel, we denote by $\mathscr{P}(E)$ the set of Radon probability measures on a given Hausdorff space $E$. The following result due to Colin de Verdi\`ere, Hillairet and Tr\'elat, which is valid for any sub-Laplacian $\Delta_{\sR}$, is the starting point of our analysis.
\begin{prop}\cite[Theorem B and Remark 1.4]{de2018spectral} \label{p:analogtheoremB}
Let $(\varphi_k)_{k\in\mathbb{N}^*}$ be an $L^2(M,\mu)$-normalized sequence of eigenfunctions of $-\Delta_{\sR}$ with eigenvalues $(\lambda_k)_{k\in\mathbb{N}^*}$ labeled in non-decreasing order and tending to $+\infty$. Let $\nu$ be a QL associated with $(\varphi_{k})_{k\in\mathbb{N}^*}$. Using the identification \eqref{e:identification}, the probability measure $\nu$ can be written as the sum 
\begin{equation} \label{e:nu0nuinfty}
\nu=\beta\nu_0+(1-\beta)\nu_\infty
\end{equation} 
of two mutually singular measures with $\nu_0,\nu_\infty \in\mathscr{P}(S^*M)$, $\beta\in[0,1]$ and
\begin{enumerate}[(1)]
\item $\nu_0(S\Sigma)=0$ and $\nu_0$ is invariant under the sub-Riemannian geodesic flow $\vec{g}^*$;
\item $\nu_\infty$ is supported on $S\Sigma$. 
\end{enumerate}
Moreover, if $(\varphi_k)_{k\in\mathbb{N}^*}$ is an orthonormal basis of $L^2(M,\mu)$, there exists an increasing sequence $(k_\ell)_{\ell\in\mathbb{N}}$ of positive integers, of density $1$, i.e.,
\[ \lim_{\alpha\rightarrow +\infty} \frac{\# \{\ell\in \N \mid k_\ell\leq \alpha\}}{\alpha}=1, \]
such that, if $\nu$ is a QL associated with a subsequence of $(\varphi_{k_\ell})_{\ell\in\mathbb{N}}$, then the support of $\nu$ is contained in $S\Sigma$, i.e., $\beta=0$ in the previous decomposition.\footnote{The proof of this last fact follows from the results in \cite{de2018spectral}, although it is not explicitely stated there. Let us sketch the proof. By \cite[Proposition 4.3]{de2018spectral}, we know that the microlocal Weyl measure of $\Delta_{\sR}$ is supported in $S\Sigma$. It then follows from \cite[Corollary 4.1]{de2018spectral} that for every $A\in\Psi^0(M)$ whose principal symbol vanishes on $\Sigma$, there holds $V(A)=0$, where $V(A)$ is the variance introduced in \cite[Definition~4.1]{de2018spectral}. Finally, following the proof of Theorem B(2) in \cite{de2018spectral}, we get the result.}
\end{prop}
The last part of Proposition \ref{p:analogtheoremB} shows that $\nu_\infty$ is the ``main part'' of the QL. In \cite{de2018spectral}, its invariance properties are determined in the following particular case (we do not recall here the definitions of three-dimensional contact sub-Laplacian and Reeb flow which can be found in \cite{de2018spectral}):
\begin{theorem}
If $\Delta_{\sR}$ is a three-dimensional contact sub-Laplacian, then $\nu_\infty$ is invariant under the lift of the Reeb flow to $S\Sigma$.
\end{theorem}
Our main results, namely Theorems  \ref{t:mainQL}, \ref{t:conversereplacement} and \ref{t:goodconversereplacement}, establish invariance properties of $\nu_\infty$ for other sub-Laplacians, showing a richer behavior than in the three-dimensional contact case.

\begin{remark}
Explicit examples of QLs for which $\beta\neq 0$ are given in \cite[Proposition 3.2(1)]{de2018spectral}.
\end{remark}

\begin{remark} \label{r:euclstrat}
In this paper, we take the Euclidean point of view, in that we do not use pseudodifferential calculus adapted to the stratified Lie algebra, a calculus that is frequently involved with sub-Laplacians. Nevertheless, our results share connexions with important problems in non-commutative Fourier analysis, as explained in Section \ref{s:bibliocomments}.
\end{remark}

\subsection{A preliminary result under a commutativity assumption}
The description of QLs for general sub-Laplacians is a difficult problem, as is the case for Riemannian Laplacians (see Section \ref{s:bibliocomments}). 

In this paper, we restrict our attention to particular sub-Laplacians, for which, despite their lack of ellipticity, techniques of joint (elliptic) spectral calculus apply in the presence of additional commutativity assumptions. 

The preliminary result which we present in this section holds under a commutativity assumption which we now introduce.

\subsubsection{The commutativity assumption}

Let us fix a sub-Laplacian $\Delta_{\sR}$ on $M$ as in \eqref{e:defsubLapl}. In this paper, we take the notation $\N=\{0,1,\ldots\}$ for the set of non-negative integers. We make the following assumption:

\medskip 

\begin{assumption} \label{a:a}There exist $m\in\N$ and $Z_1,\ldots,Z_m$ smooth global vector fields on $M$ such that:
\begin{enumerate}[(i)]
\item At any point $x\in M$ where $\mathcal{D}_x\neq T_xM$, the vector fields $Z_1(x),\ldots,Z_m(x)$ complete $\mathcal{D}_x$ into a basis of $T_xM$ (in particular, they are independent and thus do not vanish at these points);
\item For any $1\leq i,j\leq m$, there holds $[\Delta_{\sR},Z_i^*Z_i]=[Z_i^*Z_i,Z_j^*Z_j]=0$.
\end{enumerate}
\end{assumption}
Point (ii) is a strong assumption, sometimes related to the action of a group of symmetries. Assumption \ref{a:a} is satisfied for example in the following cases:
\begin{itemize}
\item For (elliptic) Laplace--Beltrami operators. In this case, $m=0$.
\item For sub-Laplacians on (quotients of) step $2$ Carnot groups. A Carnot group is a simply connected nilpotent Lie  group $G$ which is stratified of step $2$,  in the sense that its  left-invariant  Lie  algebra $\mathfrak g$, assumed to be real-valued and of finite dimension, is endowed with a vector space decomposition $\mathfrak g=\mathfrak v\oplus \mathfrak z$ where $ \mathfrak z=[\mathfrak v,\mathfrak v]\neq \{0\}$ and $[\mathfrak v,\mathfrak z]=\{0\}$. The exponential map $\exp: G\rightarrow \mathfrak g$, which is a diffeomorphism, allows to identify $G$ and $\mathfrak g$. We also assume that $\mathfrak g$ carries a scalar product $\langle \cdot,\cdot\rangle$ for which $\mathfrak v$ and $\mathfrak z$ are mutually orthogonal. There exists an orthonormal basis of left-invariant vector fields $X_1,\ldots,X_N$ of $\mathfrak v$ for $\langle \cdot,\cdot \rangle_{|\mathfrak v}$. The associated sub-Laplacian is $\Delta_{\mathfrak g}=\sum_{i=1}^N X_i^2$, which can also be defined on any compact left-quotient $H$ of $G$. Then, taking the family $(Z_j)_{1\leq j\leq m}$ as a basis of $\mathfrak z$, we see that Assumption \ref{a:a} is satisfied. This setting encompasses the case of sub-Laplacians defined on quotients of the $(2d+1)$-dimensional Heisenberg group, or more generally for H-type (Heisenberg-type) sub-Laplacians (see  \cite{Kaplan80}, \cite{fermanian2020quantum} and Appendix \ref{s:contact}).
\item For Baouendi--Grushin-type sub-Laplacians: e.g., for $\partial_x^2+\sin(x)^2\partial_y^2$ in $(\R/2\pi\Z)^2$, the set of points $x$ such that $\mathcal{D}_x\neq T_xM$ consists of the singular lines $\{x=0\}$ and $\{x=\pi\}$, and we can take $Z_1=\partial_y$. Note that the usual Baouendi--Grushin operator is $\partial_x^2+x^2\partial_y^2$, but we put here a sine in order to define it on a compact manifold without boundary.
\item For the horizontal Laplacian associated to a connection on a principal bundle over a Riemannian manifold. The $Z_i$ are then generated by the Lie algebra of the structure group and sweep out the fiber (see \cite[Section 11.2 and Appendix B]{montgomery2002tour}). Joint eigenfunctions of horizontal Laplacians on principal bundles have already been studied in several papers, see for instance \cite{guillemin90} and \cite{taylor1992semiclassical}.
\item The above examples are ``step 2,'' but it is also possible to build ad hoc sub-Laplacians satisfying Assumption \ref{a:a} and requiring higher-order brackets of the $X_i$ to generate the whole tangent bundle (see Appendix \ref{a:Martinet}).
\item For manifolds obtained as products of the previous examples (and associated sub-Laplacians obtained by sum), since Assumption \ref{a:a} is stable under product.
\end{itemize}

Assumption \ref{a:a} may be regarded as a quantum integrability assumption (see \cite{TZ03}). There is an $\R^m$ action generated by the Poisson commuting Hamiltonians $h_{Z_j}$ on $T^*M$ induced by the $Z_j$, with momentum map $(\sqrt{g^*}, |h_{Z_1}|,\ldots,|h_{Z_m}|):T^*M\rightarrow \R^{m+1}$. Therefore, the present work has to be compared with \cite{zelditch1990quantum} (notably Section 3), where the quantum limits for on- and off-diagonal matrix elements are computed in the  case where there is a torus action. The main difference is that the present work deals with a non-compact abelian action.

\subsubsection{The cotangent bundle \texorpdfstring{$T^*M$}{T*M} under Assumption \ref{a:a}} \label{s:cotangentbdl}
Our goal in this section is to describe the cotangent bundle $T^*M$ when Assumption \ref{a:a} is satisfied. Along the way, we introduce a bunch of useful notations.

In this paper, we denote by $\omega$ the canonical symplectic form on the cotangent bundle $T^*M$ of $M$. In local coordinates $(q,p)$ of $T^*M$, we have $\omega=\sum_j dp_j\wedge dq_j$ where $q=(q_1,\ldots,q_N)$ and $p=(p_1,\ldots,p_N)$. Given a smooth Hamiltonian function $h:T^*M\rightarrow \R$, we denote by $\vec h$ the corresponding Hamiltonian vector field on $T^*M$, defined by $\iota_{\vec h}\omega=-dh$. Given any smooth vector field $V$ on $M$, we denote by $h_V$ the Hamiltonian function (momentum map) on $T^*M$ associated with $V$, defined in local coordinates by $h_V(q,p)=p(V(q))$. The Hamiltonian flow $\exp(t\vec{h}_V)$ of $h_V$ projects onto the integral curves of $V$.

In the entirety of the sequel, we consider a sub-Laplacian $\Delta_{\sR}$ satisfying Assumption~\ref{a:a}. Let $\mathcal{P}$ be the set of all subsets of $\{1,\ldots,m\}$. We write $\Sigma$ as a disjoint union
\begin{equation} \label{e:SigmaJ}
\Sigma =  	\bigsqcup_{\mathcal{J}\in\mathcal{P}} \Sigma_{\mathcal{J}}
\end{equation}
where, for $\mathcal{J}\in\mathcal{P}$, $\Sigma_{\mathcal{J}}$ is defined as the set of points $(q,p)\in\Sigma$ for which
\begin{equation} \label{e:defSigmaJ}
\big\{ j\in\{1,\ldots,m\}, h_{Z_j}(q,p)\neq 0\big\}=\mathcal{J}.
\end{equation}
Note that \eqref{e:SigmaJ} is a disjoint union. Also, let us justify that the sets $\Sigma_\J$ are non-empty. We denote by $\pi:T^*M\rightarrow M$ the canonical projection. We notice that $\pi(\Sigma)=\{x\in M, \; \mathcal{D}_x\neq T_xM\}\neq \varnothing$. We pick $q\in \pi(\Sigma)$, then the non-vanishing vector fields $Z_j$ are independent at $q$ (point (i) in Assumption \ref{a:a}). Since $h_{Z_j}(q,p)=p(Z_j(q))$, we conclude that for any $\mathcal{J}\in\mathcal{P}$ there exists $p$ such that $(q,p)\in\Sigma_{\mathcal{J}}$.

\subsubsection{Quantum Limits under Assumption \ref{a:a}}

Our first preliminary result states that it is possible to split any QL into several pieces that come from well-characterized parts of the associated sequence of eigenfunctions. In practice, it will be possible to study each piece separately, and then to glue the results together. In order to give a precise statement, we need to define joint microlocal defect measures.
\begin{definition} \label{d:mdm}
Let $(u_k)_{k\in\N^*},(v_k)_{k\in\N^*}$ be bounded sequences in $L^2(M)$ such that $u_k$ and $v_k$ weakly converge to $0$ as $k\rightarrow +\infty$. A joint microlocal defect measure of $(u_k)_{k\in\N^*}$ and $(v_k)_{k\in\N^*}$ is any Radon measure $\nu_{\mathrm{joint}}$ on $S^*M$ for which there exists an extraction $\sigma:\N^*\rightarrow\N^*$ such that for any $0$-th order pseudodifferential operators $A$ with principal symbol $a=\sigma_P(A)$, there holds
\begin{equation*}
(Au_{\sigma(k)},v_{\sigma(k)})\underset{k\rightarrow +\infty}{\longrightarrow} \int_{S^*M}ad\nu_{\mathrm{joint}}.
\end{equation*}
\end{definition}
In case $u_k=v_k$ for any $k\in\N^*$, we recover the microlocal defect measures of Definition \ref{d:defmdm}. Note that joint microlocal defect measures are signed measures, and that joint QLs (defined as joint microlocal defect measures of two sequences of normalized eigenfunctions) are not necessarily invariant under the geodesic flow, even in the Riemannian case. 

The following proposition will be instrumental to the proof of our main results.
\begin{proposition} \label{t:splittinggeneral}
Let $\Delta_{\sR}$ satisfy Assumption \ref{a:a}. We assume that $(\varphi_k)_{k\in\N^*}$ is a normalized sequence of eigenfunctions of $-\Delta_{\sR}$ with associated eigenvalues $\lambda_k\rightarrow +\infty$. Then, up to extraction of a subsequence, one can decompose 
\begin{equation} \label{e:decompovarphik}
\varphi_k=\varphi_k^{\varnothing}+\sum_{\mathcal{J}\in\mathcal{P}\setminus \{\varnothing\}}\varphi_k^{\mathcal{J}},
\end{equation}
 with the following properties:
\begin{itemize}
\item The sequence $(\varphi_k)_{k\in\N^*}$ has a unique QL $\nu$;
\item For any $\J\in\myP$ and any $k\in\N^*$, $\varphi_k^{\mathcal{J}}$ is an eigenfunction of $-\Delta_{\sR}$ with eigenvalue $\lambda_k$;
\item Using the identification \eqref{e:identification}, the sequence $(\varphi_k^{\varnothing})_{k\in\N^*}$ admits a unique microlocal defect measure $\beta\nu^{\varnothing}$, where $\beta\in [0,1]$, $\nu^{\varnothing}\in\mathscr{P}(S^*M)$ and $\nu^{\varnothing}(S\Sigma)=0$;
\item For any $\mathcal{J}\in\mathcal{P}\setminus \{\varnothing\}$, the sequence $(\varphi_k^{\mathcal{J}})_{k\in\N^*}$ also admits a unique microlocal defect measure $\nu^{\mathcal{J}}$, having all its mass contained in $S\Sigma_\J$;
\item For any $\J\neq\J'\in\myP$, the joint microlocal defect measure of the sequences $(\varphi_k^\J)_{k\in\N^*}$ and $(\varphi_k^{\J'})_{k\in\N^*}$ vanishes. As a consequence,
\begin{equation} \label{e:eqdefectmeasure}
\nu=\beta\nu^{\varnothing}+\sum_{\mathcal{J}\in\mathcal{P}\setminus \{\varnothing\}}\nu^{\mathcal{J}}
\end{equation}
and the sum $(1-\beta)\nu_\infty:=\underset{\mathcal{J}\in\mathcal{P}\setminus \{\varnothing\}}{\sum}\nu^{\mathcal{J}}$ is supported in $S\Sigma$.
\end{itemize}
\end{proposition}
In \eqref{e:eqdefectmeasure}, we separated the empty set from the other subsets $\J\in\myP\setminus \{\varnothing\}$ to emphasize on the concentration of $\beta\nu^\varnothing$ on $U^*M$, while the rest of the measure $\nu$ in \eqref{e:eqdefectmeasure} is supported in $S\Sigma$. This is purely artificial, since one could have included $\beta\nu^\varnothing$ into the sum over $\J$. Besides, the notation $\nu^\varnothing$ used above corresponds to the notation $\nu_0$ in \cite{de2018spectral} (see Proposition~\ref{p:analogtheoremB} above): we changed it to get a unified notation for the different parts of the QL, namely $\nu^\varnothing$ and $\nu^\J$.

Proposition \ref{t:splittinggeneral} is proved with joint spectral calculus (see \cite[VII and VIII.5]{ReeSim72}) for the operators $Z_1^*Z_1,\ldots, Z_m^*Z_m$ and $-\Delta_{\sR}$ which is made possible thanks to Assumption \ref{a:a}. The ideas underlying Proposition \ref{t:splittinggeneral} are close to those of \cite[Theorem 0.6]{colin79} and the proof is inspired by the seminal paper \cite{gerard1991microlocal}.

\subsection{Main results on products of the three-dimensional Heisenberg group} \label{s:products}

\subsubsection{Products of the three-dimensional Heisenberg group} \label{s:productsoftheheisenberggroup}
Our main results give further information on QLs, but are restricted to a specific family of sub-Laplacians, which in particular satisfy Assumption \ref{a:a}. In order to define these operators, let us first recall the definition of the three-dimensional Heisenberg group. If we endow $\R^{3}$ with the product law
\begin{equation} \label{e:grouplaw}
(x,y,z) \star (x',y',z') =(x+x',y+y',z+z'-xy'),
\end{equation}
then, with this law, $\widetilde{\mathbf{H}}=(\R^{3},\star)$ is a Lie group, which is isomorphic to the group of matrices
\begin{equation*}
\left\{\begin{pmatrix} 1&x&-z \\ 0&1&y \\ 0&0&1\end{pmatrix}, \ x,y,z \in \R \right\}
\end{equation*}
endowed with the standard product law on matrices.

We consider the left quotient $\mathbf{H}=\Gamma\backslash\widetilde{\mathbf{H}}$ where $\Gamma=(\sqrt{2\pi}\mathbb{Z})^{2}\times 2\pi\mathbb{Z}$ is a cocompact subgroup of $\widetilde{\mathbf{H}}$ (so that $\mathbf{H}$ is compact). Note that $\mathbf{H}$ is not homeomorphic to an abelian torus since its fundamental group is $(\Gamma,\star)$, which is non-commutative. The vector fields on $\mathbf{H}$
\begin{equation*}
X=\partial_{x} \quad \text{and}  \quad Y=\partial_{y}-x\partial_z
\end{equation*}
are left invariant, and we consider 
\[ \Delta_{\mathbf{H}}=X^2+Y^2 \] 
the associated sub-Laplacian; here $\mu$ is the Lebesgue measure $\mu=dxdydz$ and $(X,Y)$ is orthonormal for $g$. 

Then, we consider the product manifold $\mathbf{H}^{m}$ and the associated sub-Laplacian $\Delta$ for some integer $m\geq 2$, that is
\begin{equation} \label{e:tracDelta}
\Delta=\Delta_{\mathbf{H}}\otimes (\text{Id})^{\otimes m-1}+\Id\otimes \Delta_{\mathbf{H}} \otimes (\Id)^{m-2}+\cdots+ (\text{Id})^{\otimes m-1}\otimes \Delta_{\mathbf{H}},
\end{equation}  
which is a second-order pseudodifferential operator. Below, we give an expression \eqref{e:Deltafields} for $\Delta$ which is more tractable. All the eigenvalues of $-\Delta$ are integers and we will describe them more precisely in Section \ref{s:spectralspectral}.
In the sequel, we fix once for all an integer $m\geq 2$. 

\begin{remark} \label{r:tensoreigen}
If $(\varphi_k)_{k\in\N^*}$ denotes an orthonormal basis of $L^2(\mathbf{H})$ consisting of eigenfunctions of $-\Delta_{\mathbf{H}}$, then 
\begin{equation*}
\{ \varphi_{k_1}\otimes \cdots\otimes \varphi_{k_m} \mid k_1, \ldots, k_m \in \N^*\}
\end{equation*}
 is an orthonormal basis of $L^2(\M)$ consisting of eigenfunctions of $-\Delta$. However, there exist orthonormal bases of $L^2(\M)$ which cannot be put in this tensorized form.
\end{remark}
In this introductory section, the sub-Laplacian we consider is either $\Delta_{\mathbf{H}}$, or $\Delta$, or an arbitrary sub-Laplacian $\Delta_{\sR}$ on a general sub-Riemannian manifold $(M,\mathcal{D},g)$. In all cases, we keep the same notations $g^*$, $\Sigma$ and $S\Sigma$ to denote the objects introduced in Section \ref{s:cotangentbdl}, without any reference in the notation to the underlying manifold even for the particular sub-Laplacians $\Delta_{\mathbf{H}}$ and $\Delta$. It should not lead to any confusion since the context is precisely stated when necessary.

In order to give a precise statement of our main results, it is necessary to introduce a decomposition of the sub-Laplacian $\Delta$ defined by \eqref{e:tracDelta}. Taking coordinates $(x_j,y_j,z_j)$ on the $j$-th copy of $\mathbf{H}$, we can write 
\begin{equation} \label{e:Deltafields}
\Delta=\sum_{j=1}^m \ (X_j^2+Y_j^2)
\end{equation}
with $X_j=\partial_{x_j}$ and $Y_j=\partial_{y_j}-x_j\partial_{z_j}$. We note that $\Delta$ satisfies Assumption \ref{a:a} (for $Z_j=\partial_{z_j}$ for $j=1,\ldots,m$).

For $1\leq j\leq m$, we consider the operator $R_j=\sqrt{\partial_{z_j}^*\partial_{z_j}}$ and we make an $L^2(\M)$ Fourier expansion with respect to the $z_j$-variable in the $j$-th copy of $\mathbf{H}$. On the eigenspaces corresponding to non-zero modes of this Fourier decomposition, we define the operator $\Omega_j=-R_j^{-1}\Delta_j=-\Delta_jR_j^{-1}$ where $\Delta_j=X_j^2+Y_j^2$. Thus, $-\Delta$ acts as
\begin{equation} \label{e:decompolaplacienintro}
-\Delta=\sum_{j=1}^m R_j\Omega_j
\end{equation}
on any eigenspace of $-\Delta$ on which $R_j\neq 0$ for any $1\leq j\leq m$. Moreover, $R_j$ and $\Omega_j$ are pseudodifferential operators of order $1$ in any cone of $T^*\M$ whose intersection with some conic neighborhood of the set $\{p_{z_j}=0\}$ is reduced to $0$ (but not near $\{p_{z_j}=0\}$ since the principal symbol $|p_{z_j}|$ of $R_j$ is not differentiable there).

The operator $\Omega_j$, seen as an operator on the $j$-th copy of $\mathbf{H}$, is an harmonic oscillator, having in particular eigenvalues $2n+1$, $n\in\mathbb{N}$ (see \cite[Section 3.1]{de2018spectral}). Moreover, the operators $\Omega_i$ (considered this time as operators on $\M$) commute with each other and with the operators $R_j$. 

\subsubsection{Flows and probabilities on \texorpdfstring{$\Sigma$}{Σ}} \label{s:flowsandprobabilities}
Let us briefly describe $\Sigma$ for the sub-Laplacian $\Delta$. Denoting by $(q,p)$ the canonical coordinates in $T^*\M$ as
\[ q=(x_1,y_1,z_1,\ldots, x_m,y_m,z_m) \] and \[ p=(p_{x_1},p_{y_1},p_{z_1},\ldots,p_{x_m},p_{y_m},p_{z_m}), \] we obtain that
\begin{equation*}
\Sigma=\big\{(q,p)\in T^*\M \ \mid \ p_{x_j}=p_{y_j}-x_jp_{z_j}=0 \text{ for any } 1\leq j\leq m\big\}.
\end{equation*}
The map \begin{align*}\Sigma&\rightarrow \M\times\R^m\\ (q,p)&\mapsto (q,p_{z_1},\ldots,p_{z_m})\end{align*} is one-to-one.  Above any point $q\in\M$, the fiber of $\Sigma$ is of dimension $m$, and therefore, above any point $q\in\M$, $S\Sigma$ consists of an $(m-1)$-dimensional sphere. At some point in Section \ref{s:converse}, we will consider the coordinates $(q,p_{z_1},\ldots,p_{z_m})$ on $\Sigma$ and the coordinates $(q,[p_{z_1}:\cdots:p_{z_m}])$ on $S\Sigma$, where the notation $[p_{z_1}:\cdots:p_{z_m}]$ stands for homogeneous coordinates.

Writing $\Sigma$ as a disjoint union \eqref{e:SigmaJ}, we notice that $\Sigma_{\mathcal{J}}$ is the set of points $(q,p)\in\Sigma$ with $p=(p_{x_1},p_{y_1},p_{z_1},\ldots,p_{x_m},p_{y_m},p_{z_m})$ such that 
\[ \left\{j\in\{1,\ldots,m\}, p_{z_j}\neq 0\right\}=\mathcal{J}. \]
The notation $S\Sigma_\J$ designates in the sequel the set of points $(q,p)$ of $S\Sigma$ which have null (homogeneous) coordinate $p_{z_i}$ for any $i\notin \J$ and non-null $p_{z_j}$ for $j\in\J$. Note that this set is, in general, neither open nor closed.

For $\J\in\mathcal{P}\setminus \{\varnothing\}$, we consider the simplex
\begin{equation*}
\mathbf{S}_\J= \bigg\{s=(s_j)\in \R_+^\J, \ \sum_{j\in\J}s_j=1\bigg\}
\end{equation*}
and, for $s=(s_j)\in \mathbf{S}_\J$ and $(q,p)\in \Sigma_\J$, we set
\begin{equation*}
\rho_s^\J(q,p)=\sum_{j\in\J}s_j|p_{z_j}|.
\end{equation*}
We denote by  $\sigma_P$ the principal symbol (see Appendix \ref{a:pseudo}). We have
\begin{equation} \label{e:defRs}
\rho^\J_s(q,p)=(\sigma_P(R_s))_{|\Sigma_\J} \ \text{   where   } \ \ R_s =\sum_{j\in\J}s_jR_j,
\end{equation}
noting that $R_s$ is a pseudodifferential operator in $\Sigma_\J$.
 Moreover, the vector field
\begin{equation}\label{e:vecrhoJs}
\vec\rho^\J_s=\sum_{j\in\J}{\text{sgn}}(p_{z_j})s_j\partial_{z_j}.
\end{equation}
 is well-defined on $\Sigma_{\mathcal{J}}$ and smooth.\footnote{Roughly speaking, $\vec{\rho}_s^\J$ is some kind of Hamiltonian vector field associated to $\rho_s^\J$, but note that $\Sigma_\J$ is not necessarily a symplectic manifold since it may be odd-dimensional, hence the term ``Hamiltonian'' is not meaningful here.}
\begin{remark} \label{r:flowrho}
All coordinates $x_i,y_i$ (for $1\leq i\leq m$) and $z_i$ (for $i\notin\J$) are preserved under the flow of $\vec\rho^\J_s$. Once fixed these coordinates, any trajectory of the flow of $\vec\rho^\J_s$ is conjugated to a geodesic trajectory in the flat $|\J|$-dimensional Euclidean torus. This trajectory depends only on $s$ and the signs of the $p_{z_j}$, which are preserved by the flow.
\end{remark}

Finally, we introduce a set of probability measures on $S^*\M$ having specific invariance properties: 
\begin{equation}
\begin{split}
\mathscr{P}_\infty&=\bigg\{\nu_\infty \in \mathscr{P}(S^*\M) \text{ which can be written as } \nu_\infty  =\sum_{\mathcal{J}\in\mathcal{P}\setminus \{\varnothing\}}\int_{\mathbf{S}_\J}\nu^\J_sdQ^\J(s), \\
&\qquad \text{where for any } \J\in \myP, \ Q^\J \text{ is a non-negative Radon measure on } \mathbf{S}_\J,\smallskip\\
&\qquad \text{and } \forall \J\in \myP, \forall s\in \mathbf{S}_\J, \ \nu_s^\J \in \mathscr{P}(S^*\M), \  \nu_s^\J(S^*\M\setminus S\Sigma_\J)=0,  \\
&\qquad \text{and for $Q^\J$-almost any $s\in\mathbf{S}_\J$ }, \nu^\J_s \text{ is invariant under } \vec\rho^\J_s\bigg\} .
\end{split}
\label{e:Pssigma}
\end{equation}
This means that for any continuous function $a:S\Sigma\rightarrow \R$, there holds 
\begin{equation*}
\int_{S\Sigma} ad\nu_\infty=\sum_{\J\in\mathcal{P}\setminus \{\varnothing\}}\int_{\mathbf{S}_\J} \bigg(\int_{S\Sigma_\J}ad\nu^\J_s\bigg) dQ^\J(s).
\end{equation*}
Any measure $\nu_\infty\in\mathscr{P}_\infty$ is supported in $S\Sigma$, and its invariance properties are given separately on each set $S\Sigma_\J$ (for $\J\in\mathcal{P}\setminus \{\varnothing\}$). Its restriction to any of these sets, denoted in the sequel by 
\[\nu^\J=\int_{\mathbf{S}_\J}\nu^\J_sdQ^\J(s), \] 
can be disintegrated with respect to $\mathbf{S}_\J$, and for $Q^\J$-almost any $s\in\mathbf{S}_\J$, there is a corresponding measure $\nu_s^\J$ which is invariant under the flow $e^{t\vec\rho^\J_s}$.

\subsubsection{Main results} \label{s:mainresultssection}

Our first main result is the following.
\begin{theorem} \label{t:mainQL} \label{t:qlqe} Let $(\varphi_k)_{k\in\mathbb{N}^*}$ be an $L^2(\M)$-normalized sequence of eigenfunctions of $-\Delta$ associated with the eigenvalues $\lambda_k\rightarrow +\infty$. Let $\nu$ be a QL associated to the sequence $(\varphi_k)_{k\in\mathbb{N}^*}$. Then the measure $\nu_\infty$ defined in \eqref{e:nu0nuinfty} satisfies $\nu_\infty\in \mathscr{P}_\infty$ where $ \mathscr{P}_\infty$ has been introduced in \eqref{e:Pssigma}.
\end{theorem}

Note that Theorem \ref{t:mainQL} holds for \emph{any} $L^2(\M)$-normalized sequence of eigenfunctions of $-\Delta$, and not only for the bases described in Remark \ref{r:tensoreigen}. 

We were not able to prove that all elements of $\mathscr{P}_\infty$ can be realized as a QL, which would be some kind of converse of Theorem \ref{t:mainQL}. We do not know whether it is true. However, we were able to prove two results in this direction.

The first one realizes a family of probability measures strictly included in $\mathscr{P}_\infty$ as QLs. Any element in this family is obtained as the tensorial product of two measures: $\J$ being fixed, the first measure is a kind of Lebesgue measure in the copies of $\mathbf{H}$ corresponding to $i\notin\J$, and the second measure is a measure ``invariant in the $z_j$ variable'' in the copies of $\mathbf{H}$ corresponding to $j\in\J$.

To make it rigorous, we denote by $\mathbf{H}^\J$ (resp. $\mathbf{H}^{\notin\J}$) the product of copies of $\mathbf{H}$ with variables $x_j,y_j,z_j$, for $j\in\J$ (resp. $j\notin \J$). We define $\mathscr{M}^\J$ as the set of Radon probability measures on $S^*\mathbf{H}^\J$ which are invariant under $\partial_{z_j}$ for $j\in\J$. Then, we define the probability measure $\ell^{\notin \J}$ on  $T^*\mathbf{H}^{\notin \J}$ as the tensorial product of the Haar measure on $\mathbf{H}^{\notin \J}$ and the Dirac mass on the zero section in the fibers. And finally we set
\begin{equation}\label{e:defDJ}
\mathscr{D}^\J=\{m^\J\otimes\ell^{\notin \J}  \mid  m^\J\in\mathscr{M}^\J\}
\end{equation}
which is viewed as a set of Radon probability measures on $S^*\mathbf{H}^m$.

Our second main result is the following.
\begin{theorem}\label{t:conversereplacement}
For any $\J\in\myP\setminus\{\varnothing\}$, let $\nu^\J\in\D^\J$, and let $c_\J\geq 0$ so that 
\[ \sum_{\J\in\myP\setminus\{\varnothing\}}c_\J=1. \] 
Then
\[ \nu=\sum_{\J\in\myP\setminus\{\varnothing\}}c_\J\nu^\J \]
is a QL.
\end{theorem}
Theorem \ref{t:conversereplacement} has the drawback that any measure $\nu$ as in the statement is invariant under all vector fields $\partial_{z_j}$ at the same time, and thus Theorem \ref{t:conversereplacement} does not prove the existence of QLs which are invariant under a single flow $\vec{\rho}_s^\J$. Our last result shows that such QLs indeed exist:
\begin{theorem}\label{t:goodconversereplacement}
If $m\geq 2$, there exists a QL $\nu$ such that the equation $\vec{\rho}_s^\J\nu=0$ is satisfied only for a unique $\J\in\myP\setminus\{\varnothing\}$ and a unique $s\in\mathbf{S}_\J$.
\end{theorem}
This last result shows that all vector fields $\vec{\rho}_s^\J$ play a role at the quantum level.

\subsubsection{Comments on the main results} \label{s:comments}

\paragraph{Spectrum of $-\Delta$.} The particularly rich structure of the QLs of the sub-Laplacian $-\Delta$ described in Theorem \ref{t:mainQL} is due to the high degeneracy of its spectrum. To make an analogy with the Riemannian case, the QLs of the usual flat Riemannian torus $\mathbb{T}^2=\R^2/\Z^2$ have a rich structure (see \cite{jakobson1997quantum}), whereas the eigenfunctions and the QLs of irrational Riemannian tori are simply obtained as tensor products.

Recall that the spectrum $\text{spec}(-\Delta_{\mathbf{H}})$ is given by 
\begin{align*}
&\mathrm{spec}(-\Delta_{\mathbf{H}})= \\
&\quad \{\lambda_{\ell,\alpha}=(2\ell +1)|\alpha| \mid \ell \in \N, \ \alpha\in \N^*\} \cup \{\mu_{k_1,k_2}=2\pi(k_1^2+k_2^2) \mid (k_1,k_2)\in\Z^2\} 
\end{align*}
where $\lambda_{\ell,\alpha}$ is of multiplicity $|\alpha|$, multiplied by the number of decompositions of $\lambda_{\ell,\alpha}$ into the form $(2\ell'+1)|\alpha'|$ (see \cite[Corollary 3.3]{Fol04}, \cite[Proposition 3.1]{de2018spectral}). Therefore, using a tensorial orthonormal basis of $L^2(\M)$ consisting of eigenfunctions of $-\Delta$, we get that
\begin{align*}
&\mathrm{spec}(-\Delta)= \\
&\quad\bigg\{\sum_{j=1}^{J} \left(2n_j+1\right)|\alpha_j|+2\pi\sum_{i=1}^{2(m-J)}k_i^2\text{  with } 0\leq J\leq m, \ k_i\in \Z, \ n_j\in\N, \ \alpha_j\in \N^* \bigg\} 
\end{align*}
(see Section \ref{s:converse} for a detailed proof) and the multiplicities in $\text{spec}(-\Delta)$ can be deduced from those in $\text{spec}(-\Delta_{\mathbf{H}})$. For a description of the eigenfunctions of $\Delta_{\mathbf{H}}$, see \cite[Section 3]{Fol04}; the eigenfunctions of $\Delta$ are sums of tensor products of these eigenfunctions. Note that the eigenvalues for which $J=m$ form a density-one subsequence of all eigenvalues labeled in increasing order.

The specific algebraic structure of $\text{spec}(-\Delta)$ will be exploited in particular to prove Theorems \ref{t:conversereplacement} and \ref{t:goodconversereplacement}. 

\begin{remark}
Contrarily to those of flat tori (see \cite{jakobson1997quantum}), the QLs of $\M$ (or, more precisely, their pushforward under the canonical projection onto $\M$) are not necessarily absolutely continuous. This fact has already been noticed in the case $m=1$ in \cite[Proposition 3.2(2)]{de2018spectral} -- in this case the Dirac measure on a Reeb orbit is a (projected) QL. This can be understood as follows: on flat tori the microlocal defect measures of joint eigenfunctions are Lebesgue measures on phase space tori, which project without singularities to the base. But for $\M$, since there exist Hermite eigenfunctions which concentrate on closed orbits, the associated QLs have singular projections.
\end{remark}

\begin{remark}
There is no clear link of our result with the concept of ``second microlocalization,'' although such a link may seem possible at first sight. Focusing on a QL supported in $S\Sigma$, our study builds upon a \emph{spectral} decomposition of it, and not upon a second direction of microlocalization as is usually done while studying fine properties of sequences of solutions of an operator (see for example \cite{fermanian2000mesures}).
\end{remark}

\subsection{Related problems and bibliographical comments.}  \label{s:bibliocomments}

\paragraph{Quantum Limits of Riemannian Laplacians.} The study of QLs for Riemannian Laplacians is a long-standing question. Over the years, a particular attention has been drawn towards Riemannian manifolds whose geodesic flow is ergodic since in this case, up to extraction of a density-one subsequence, the set of QLs is reduced to the Liouville measure, a phenomenon which is called Quantum Ergodicity (see for example \cite{shnirel1974ergodic}, \cite{de1985ergodicite}, \cite{zelditch1987uniform}). For compact arithmetic surfaces, a detailed study of invariant measures lead to the resolution of the Quantum Unique Ergodicity conjecture for these manifolds, meaning that the extraction of a density-one subsequence in the previous result is even not necessary for these particular manifolds (\cite{lindenstrauss2006invariant}). In manifolds which have a degenerate spectrum, the set of QLs is generally richer: see for example  \cite{jakobson1997quantum} for the description of QLs on flat tori or \cite{anantharaman2016wigner} for the case of the disk. Also, the QLs of the sphere $\mathbb{S}^d$ equipped with its canonical metric (see \cite{jakobson1996classical}) have been fully characterized. However, to the author's knowledge, few papers until now have been devoted to the study of QLs of product of Riemannian manifolds (see \cite{anantharaman2015semiclassical}, \cite[Corollary 2]{humbert2020observability}, \cite{arnaiz2022} for recent results).

\paragraph{Quantum Limits of sub-Laplacians.} The understanding of QLs of general sub-Laplacians remains a largely unexplored question. Their study was undertaken in the work \cite{de2018spectral}, which was mainly devoted to the three-dimensional contact case -- encompassing for example the case of the manifold $\mathbf{H}$ -- although some results are valid for any sub-Laplacian (see Proposition \ref{p:analogtheoremB} of the present paper). The authors proved Weyl laws (i.e., results ``in average'' on eigenfunctions), a result of decomposition of QLs, and also Quantum Ergodicity properties (i.e., equidistribution of QLs under an ergodicity assumption) for three-dimensional contact sub-Laplacians. The QLs of H-type (or Heisenberg-type) sub-Laplacians were also implicitly studied in \cite{fermanian2020quantum} thanks to a detailed study of the Schr\"odinger flow: the authors developed a notion of semiclassical measures adapted to ``Heisenberg type'' sub-Laplacians thanks to non-commutative Fourier analysis and a subsequent adapted definition of pseudodifferential operators. Taking in Theorem 2.10(ii)(2) of \cite{fermanian2020quantum} eigenfunctions of the sub-Laplacian as initial data of the Schr\"odinger equation yields a decomposition of QLs which may be regarded as an analog of Theorem \ref{t:mainQL} in the context of H-type groups (more precisely, one should use the adaptation to the compact (quotient) setting of these results which was done in \cite{fermanian2020observability}, among other things); however, the result of \cite{fermanian2020quantum} is proved by totally different techniques, and in particular the splitting of QLs which we obtain through joint spectral calculus (see below) is replaced in \cite{fermanian2020quantum} by non-commutative harmonic analysis.

\paragraph{Non-commutative harmonic analysis.} As already mentioned in Remark \ref{r:euclstrat}, it is possible to use the stratified Lie algebra structure to study the spectral theory of (nilpotent) sub-Laplacians, as done for example in \cite{fermanian2020quantum}. This work builds upon non-commutative harmonic analysis (see \cite{taylor1986}) to develop a pseudodifferential calculus and semiclassical tools ``naturally attached to the sub-Laplacian''. It is likely that one could have given a proof of Theorems \ref{t:mainQL}, \ref{t:conversereplacement} and \ref{t:goodconversereplacement} based on similar tools as in \cite{fermanian2020quantum}. The point of view we adopt in the present paper is different: it only requires ``classical'' pseudodifferential calculus  (briefly recalled in Appendix \ref{a:pseudo}) since there is still enough commutativity and ellipticity from the choice of operators under study. Beside making the results more accessible to some readers, it allows us to isolate in each eigenfunction the piece which is responsible, in the high-frequency limit, for a given part of the QL. Moreover, our method only builds upon abstract commutation arguments, at least for Proposition \ref{t:splittinggeneral}, and in particular it avoids the computation of irreducible representations which are always specific to certain families of groups (e.g., H-type groups in \cite{fermanian2020quantum}). 

Part of our results can be reinterpreted through the light of noncommutative harmonic analysis. For example, the part of the QL in $U^*M$, namely $\beta\nu^\varnothing$ (see \eqref{e:eqdefectmeasure}), is described in \cite{fermanian2020quantum} as the part of the semiclassical measure supported above the finite dimensional representations $\pi^{0,\omega}_x$ (see \cite[Section 2.2.1]{fermanian2020quantum}), and the fact that $\beta\nu^\varnothing=0$ for ``almost all'' QLs (see Proposition \ref{p:analogtheoremB}) can be recovered from the fact that the Plancherel measure denoted by $|\lambda|^dd\lambda$ in \cite{fermanian2020quantum} gives no mass to finite-dimensional representations. 

Also, in the setting covered by Theorems \ref{t:mainQL}, \ref{t:conversereplacement} and \ref{t:goodconversereplacement}, i.e., products of quotients of the Heisenberg group, the joint spectrum of $(\Delta_1,\ldots,\Delta_m,i^{-1}\partial_{z_1},\ldots,i^{-1}\partial_{z_m})$, which can be drawn in $\R^{2m}$, is called ``Heisenberg fan''. This terminology was introduced in \cite{strichartz91} for the three-dimensional Heisenberg sub-Laplacian; in our case, this fan consists in a discrete set of points which can be gathered into lines (see \cite[Figure 1]{strichartz91}). In case $m=1$, the subset of points (or joint eigenvalues) corresponding to $\varphi_k^\varnothing$ and $\nu^\varnothing$ in the statement of Theorem \ref{t:mainQL} can be seen as points close to the vertical line $\{0\}\times\R\subset \R^2$. Similar descriptions can be given in case $m\geq 2$. Also, let us mention that we could derive from the proof of Proposition \ref{t:splittinggeneral} a generalization of the definition of the Heisenberg fan to any sub-Laplacian satisfying Assumption \ref{a:a}, as the joint spectrum of $(-\Delta_{\sR},\sqrt{Z_1^*Z_1},\ldots,\sqrt{Z_m^*Z_m})$.

Let us also mention that sub-Laplacians on products of Heisenberg groups (and, more generally, on ``decomposable groups'') were analysed in \cite{bahouri2016} with a non-commutative harmonic analysis point of view in order to establish Strichartz estimates (see notably \cite[Section 1.4 and Corollary 1.6]{bahouri2016}).

\paragraph{Joint spectral calculus.} A key ingredient in the proof of all results of the present paper is the joint spectral calculus (see \cite[VII and VIII.5]{ReeSim72} and \cite{colin79}) associated to the operators $Z_1^*Z_1,\ldots,Z_m^*Z_m$ and $-\Delta_{\sR}$. This joint calculus, at least for Heisenberg groups, is well-known, see for example \cite[Section 2]{deninger84}, or \cite{thangavelu2009} for the quotient case. It was used for instance in \cite{MRS95} to prove a Marcinkiewicz multiplier theorem in H-type groups.

\paragraph{Structure of the paper.} In Section \ref{s:parab} we prove Proposition \ref{t:splittinggeneral} using joint spectral calculus. Section \ref{s:prelimmainQL} is devoted to preliminary steps in the proof of Theorem \ref{t:mainQL}. Building upon Proposition \ref{t:splittinggeneral}, we establish Theorem \ref{t:mainQL} in Section \ref{s:prooftheo}. In Section \ref{s:converse}, we prove Theorem \ref{t:conversereplacement} by constructing explicitly a sequence of eigenfunctions with prescribed QL. In Section \ref{s:goodconverse}, we prove Theorem \ref{t:goodconversereplacement}.

In Appendix \ref{a:pseudo}, we recall some basic facts of pseudodifferential calculus and a related elementary lemma. In Appendix \ref{a:supplementary}, we build an example of step $3$ sub-Laplacian satisfying Assumption \ref{a:a}. Finally, in Appendix \ref{s:contact}, we prove a result concerning QLs of flat contact manifolds in any dimension: for such manifolds, the invariance properties of QLs are essentially the same as in the three-dimensional case. Although this is a direct consequence of the results in \cite{fermanian2020quantum}, we decided to provide here a short and self-contained proof since this can be seen as a toy model for the averaging techniques used repeatedly in the proof of Theorem~\ref{t:mainQL}.

\section{Proof of Proposition \ref{t:splittinggeneral}}  \label{s:parab}

\subsection{Notation}
We fix a sub-Laplacian $\Delta_{\sR}$ satisfying Assumption \ref{a:a}, we fix $(\varphi_k)_{k\in\mathbb{N}^*}$ a sequence of eigenfunctions of $-\Delta_{\sR}$ associated with the eigenvalues $(\lambda_k)_{k\in\mathbb{N}^*}$ with $\lambda_k\rightarrow +\infty$ and $\lVert\varphi_k\rVert_{L^2}=1$, and, possibly after extraction of a subsequence, we assume that $(\varphi_k)_{k\in\mathbb{N}^*}$ has a unique QL~$\nu$.

Let us first give an intuition of how the proof goes. Set
\begin{equation} \label{e:defE}
E=\text{Id}-\Delta_{\sR}+\sum_{j=1}^mZ_j^*Z_j\in \Psi^2(M).
\end{equation}
We decompose $\varphi_k$ as a sum of functions which are joint eigenfunctions of $-\Delta_{\sR}$ and of all the $Z_j^*Z_j$ for $1\leq j\leq m$. Thus they are also eigenfunctions of $E$. Each of these functions is an eigenfunction of $-\Delta_{\sR}$ with same eigenvalue $\lambda_k$ as $\varphi_k$. Then, roughly speaking, we gather some of these functions into  $\varphi_k^\varnothing$ or into $\varphi_k^{\mathcal{J}}$ for some $\mathcal{J}\in\mathcal{P}\setminus \{\varnothing\}$, depending on their eigenvalues with respect to the operators $Z_j^*Z_j$ (for $1\leq j\leq m$) and $-\Delta_{\sR}$. 

Fix $\mathcal{J}\in\mathcal{P}\setminus \{\varnothing\}$. The functions we select (asymptotically as $k\rightarrow +\infty$) to be in $\varphi_k^{\mathcal{J}}$ are those such that the following spectral inequalities are satisfied:
\begin{enumerate}
\item $-\Delta_{\sR}\ll E$;
\item if $i\notin \J$, then $Z_i^*Z_i\ll E$;
\item if $j\in\J$, then $Z_j^*Z_j\gtrsim E$.
\end{enumerate}
Here, since we consider joint eigenfunctions of $-\Delta_{\sR}$, $E$ and $Z_j^*Z_j$ for any $1\leq j\leq m$, the above notation $A\ll B$ (resp. $A\gtrsim B$) means that as $k\rightarrow +\infty$ the ratio between the eigenvalue with respect to $A$ and the eigenvalue with respect to $B$ tends to $0$ (resp. is bounded below).

Before proceeding towards a rigorous proof, we introduce a few notations.
The principal symbol of $E$ is
\begin{equation*}
\sigma_P(E)=g^*+\sum_{j=1}^m\sigma_P(Z_j^*Z_j)
\end{equation*}
hence $E$ is elliptic, thanks to point (i) in Assumption \ref{a:a}.
For $n\in\N^*$, let $\chi_n\in C_c^\infty(\R,[0,1])$ such that $\chi_n(x)=1$ for $|x|\leq \frac{1}{2n}$ and $\chi_n(x)=0$ for $|x|\geq \frac1n$.  Thanks to functional calculus (see \cite[VII and VIII.5]{ReeSim72}), for $\mathcal{J}\in \mathcal{P}\setminus \{\varnothing\}$, the operator
\begin{equation} \label{e:defPn}
P_n^\mathcal{J}=\chi_n\Big(\frac{\text{Id}-\Delta_{\sR}}{E}\Big)\prod_{i\notin \mathcal{J}}\chi_n\Big(\frac{Z_i^*Z_i}{E}\Big)\prod_{j\in \mathcal{J}}(1-\chi_n)\Big(\frac{Z_j^*Z_j}{E}\Big)
\end{equation}
is well-defined. Note that, thanks to point (ii) in Assumption \ref{a:a}, we know that $E$ commutes with $Z_j^*Z_j$, for any $1\leq j\leq m$, and with $-\Delta_{\sR}$, which explains why we are allowed to use the quotients of operators in \eqref{e:defPn}. 
Similarly, we consider
\begin{equation} \label{e:defPnempty}
P_n^\mathcal{\varnothing}=(1-\chi_n)\bigg(\frac{\text{Id}-\Delta_{\sR}}{E}\bigg)+\chi_n\bigg(\frac{\text{Id}-\Delta_{\sR}}{E}\bigg)\prod_{i=1}^m \chi_n\bigg(\frac{Z_i^*Z_i}{E}\bigg).
\end{equation}
We note that for any $n\in\N$, 
\begin{equation}\label{e:sumisidentity}
\sum_{\J\in\myP} P_n^\J=\text{Id}.
\end{equation}

\subsection{A preliminary lemma}\label{s:proofpropofPnJ}
\begin{lemma}\label{l:propofPnJ}
For any $\mathcal{J}\in\mathcal{P}$, the following properties hold:
\begin{enumerate}[(1)]
\item $P_n^{\mathcal{J}}\in \Psi^0(M)$;
\item $[P_n^{\mathcal{J}},\Delta_{\sR}]=0$;
\item If $\J\neq\varnothing$, then $\sigma_P(P_n^{\mathcal{J}})\rightarrow \mathbf{1}_{\Sigma_{\mathcal{J}}}$ pointwise as $n\rightarrow +\infty$, where $\mathbf{1}_{\Sigma_{\mathcal{J}}}$ is the characteristic function of $\Sigma_{\mathcal{J}}$ (defined in \eqref{e:SigmaJ}). \\
If $\J=\varnothing$, then $\sigma_P(P_n^{\mathcal{J}})\rightarrow \mathbf{1}_{U^*M}$ pointwise as $n\rightarrow +\infty$,  where $\mathbf{1}_{U^*M}$ is the characteristic function of $U^*M$.
\end{enumerate}
\end{lemma}

\begin{proof}
Let us prove Point (1). Since $E\in \Psi^2(M)$ is elliptic, it is invertible, and thus
\begin{equation*}
(\text{Id}-\Delta_{\sR})E^{-1}=E^{-1}(\text{Id}-\Delta_{\sR})\in\Psi^0(M)
\end{equation*}
is self-adjoint. Hence, by \cite[Theorem 1(ii)]{hassell2000symbolic}, $(1-\chi_n)\Big(\frac{\text{Id}-\Delta_{\sR}}{E}\Big)\in\Psi^0(M)$ with principal symbol
\begin{equation*}
(1-\chi_n)\Big(\frac{g^*}{\sigma_P(E)}\Big).
\end{equation*}
Similarly, the operators $\chi_n\Big(\frac{\text{Id}-\Delta_{\sR}}{E}\Big)$, $\chi_n\Big(\frac{Z_i^*Z_i}{E}\Big)$ and $(1-\chi_n)\Big(\frac{Z_j^*Z_j}{E}\Big)$ (for any $1\leq i,j\leq m$) belong to $\Psi^0(M)$ with respective principal symbols
\begin{equation*}
\chi_n\Big(\frac{g^*}{\sigma_P(E)}\Big), \quad  \chi_n\Big(\frac{|h_{Z_i}|^2}{\sigma_P(E)}\Big) \quad \text{and} \quad (1-\chi_n)\Big(\frac{|h_{Z_j}|^2}{\sigma_P(E)}\Big).
\end{equation*}
Hence, $P_n^{\mathcal{J}}\in\Psi^0(M)$.

Point (2) is an immediate consequence of functional calculus, since $\Delta_{\sR}$ commutes with $E$ and with $Z_j^*Z_j$ for any $1\leq j\leq m$.

Let us prove Point (3). For $\kappa>0$, we consider the cone
\begin{equation*} 
S_\kappa:= \bigg\{\frac{g^*}{\sigma_P(E)}\leq \kappa\bigg\}\; \subset T^*M
\end{equation*}
and, for $1\leq j\leq m$, we also consider the cone
\begin{equation*}
T_\kappa^j=\bigg\{\frac{|h_{Z_j}|^2}{\sigma_P(E)}\leq \kappa\bigg\} \; \subset T^*M.
\end{equation*} 

For the moment, we assume $\J\neq\varnothing$. Then, the support of $\sigma_P(P_n^\mathcal{J})$  is contained in $S_{\frac1n}$, in $T_{\frac1n}^i$ for $i\notin\J$ and in the complementary set $(T_{\frac{1}{2n}}^j)^c$ for $j\in\J$. It follows that, in the limit $n\rightarrow +\infty$, $\sigma_P(P_n^\J)$ vanishes everywhere outside the set of points $(q,p)$ satisfying $g^*(q,p)=0$,
\begin{gather}
h_{Z_i}(q,p)=0, \ \ \forall i\notin\mathcal{J} \nonumber \\
h_{Z_j}(q,p)\neq 0, \ \ \ \forall j\in \mathcal{J}. \nonumber
\end{gather}
We note that these relations exactly define the set $\Sigma_{\mathcal{J}}$. \\
Conversely, let $(q,p)\in \Sigma_{\mathcal{J}}$. Our goal is to show that $\sigma_P(P^{\mathcal{J}}_n)(q,p)=1$ for sufficiently large $n\in\N^*$. It follows from a separate analysis of the principal symbol of each factor in the product \eqref{e:defPn}: 
\begin{itemize}
\item Since $(q,p)\in\Sigma$, there holds $g^*(q,p)=0$, hence
\begin{equation*}
\chi_n\Big(\frac{g^*}{\sigma_P(E)}\Big)=1;
\end{equation*}
\item For $i\notin \mathcal{J}$, since $h_{Z_i}(q,p)=0$, there holds
\begin{equation*}
\chi_n\Big(\frac{|h_{Z_i}|^2}{\sigma_P(E)}\Big)(q,p)=1;
\end{equation*}
\item For $j\in \mathcal{J}$, we know that $h_{Z_j}(q,p)\neq 0$. Hence, for $n$ sufficiently large, at $(q,p)$, 
\begin{equation*}
(1-\chi_n)\Big(\frac{|h_{Z_j}|^2}{\sigma_P(E)}\Big)(q,p)=1.
\end{equation*}
\end{itemize}
All in all, $\sigma_P(P_n^\mathcal{J})(q,p)=1$ for sufficiently large $n$, which proves Point (3) in case $\J\neq \varnothing$. 
For the proof in the case $\J=\varnothing$, we note that by definition of $E$ \eqref{e:defE}, we have 
\[ \chi_n\Big(\frac{\text{Id}-\Delta_{\sR}}{E}\Big)\prod_{i=1}^m \chi_n\Big(\frac{Z_i^*Z_i}{E}\Big)=0 \]
as soon as $n\geq m+1$. The rest of the proof for $\J=\varnothing$ is very similar to the case $\J=\varnothing$, for the sake of brevity we do not repeat it here.
\end{proof}

\subsection{Proof of Proposition \ref{t:splittinggeneral}}\label{s:proofofproposition}
We finally prove Proposition \ref{t:splittinggeneral}. We consider, for fixed $n\in\N$ and $\J\in\mathcal{P}$, the sequence $(P_n^\J\varphi_{k})_{k\in\N^*}$, which, thanks to Point (2) of Lemma \ref{l:propofPnJ}, is also a sequence of eigenfunctions of $-\Delta_{\sR}$ with the same eigenvalues as $\varphi_k$.
For any $A\in\Psi^0(M)$, using that $P_n^\J$ is self-adjoint, there holds
\begin{align*}
(AP_n^\J\varphi_{k},P_n^\J\varphi_{k}) =(P_n^\J AP_n^\J\varphi_{k},\varphi_{k})\underset{k\rightarrow +\infty}{\longrightarrow} \int_{S^*M}\sigma_P(P_n^\J)^2 \sigma_P(A)d\nu.
\end{align*}
Hence $(P_n^\J\varphi_{k})_{k\in\N^*}$ has a unique microlocal defect measure $\nu_n^\J=\sigma_P(P_n^\J)^2 \nu$.
Finally, we take $\nu^\J$ a weak-star limit of $(\nu^\J_n)_{n\in\N}$ and $\beta\nu^\varnothing$ a weak-star limit of $(\nu_n^\varnothing)_{n\in\N}$, with $\nu^\varnothing\in\mathscr{P}(S^*M)$ and $\beta\in [0,1]$. Up to successive extractions we can assume that all these weak-star limits are obtained with the same extraction $\sigma_1:\N^*\rightarrow \N^*$.

\begin{lemma} \label{l:supportnuJemptyset}
There holds $\nu^\varnothing(S\Sigma)=0$ and, for $\J\in\myP\setminus\{\varnothing\}$, $\nu^\J$ gives no mass to the complement of $S\Sigma_\J$ in $S^*M$. 
\end{lemma}
\begin{proof}
For $\J\in\myP$ (possibly $\J=\varnothing$) and any $A\in\Psi^0(M)$, using that $P_{\sigma_1(n)}^\J$ is self-adjoint, 
\begin{align*}
\int_{S^*M}\sigma_P(A)d\nu_{\sigma_1(n)}^{\mathcal{J}}&=\int_{S^*M}\sigma_P(P_{\sigma_1(n)}^\J)^2 \sigma_P(A)d\nu\\
&\underset{n\rightarrow+\infty}{\longrightarrow}\begin{cases} \int_{S^*M} \sigma_P(A)\mathbf{1}_{S\Sigma_\J}d\nu \text{ if $\J\neq\varnothing$}\\ \int_{S^*M} \sigma_P(A)\mathbf{1}_{U^*M}d\nu \text{ if $\J=\varnothing$}\end{cases}
\end{align*}
by the dominated convergence theorem and Lemma \ref{l:propofPnJ}, which proves the result.
\end{proof}

Let us summarize the situation: there exists an extraction $\sigma_1:\N^*\rightarrow \N^*$ such that for any $a\in\mathscr{S}_{\text{hom}}^0(M)$ (see Appendix \ref{a:pseudo}),
\begin{equation}\label{e:convsigma1}
\int_{S^*M}a d\nu_{\sigma_1(n)}^\J\underset{n\rightarrow +\infty}{\longrightarrow}\begin{cases}
\int_{S^*M}a d\nu^\J \text{ if $\J\neq \varnothing$} \\
\int_{S^*M}a \beta d\nu^\varnothing \text{ if $\J= \varnothing$}
\end{cases}
\end{equation}
and for any $n\in\N^*$ and any $A\in\Psi^0(M)$ with principal symbol $a$,
\begin{equation}\label{e:convink}
(AP_{\sigma_1(n)}^\J\varphi_k, P_{\sigma_1(n)}^\J\varphi_k)\underset{k\rightarrow +\infty}{\longrightarrow}\int_{S^*M}a d\nu_{\sigma_1(n)}^\J.
\end{equation}
Choosing first $n$ large, and then $k$ large, the combination of \eqref{e:convsigma1} and \eqref{e:convink} yields the existence of a function $r$ tending to $+\infty$ at $+\infty$ with $r(k)\ll k$ at $+\infty$ such that $P_{r(k)}^\J\varphi_{k}$ has a unique microlocal defect measure which is $\nu^\J$ for $\J\neq\varnothing$ and $\beta\nu^\varnothing$ for $\J=\varnothing$. 

Setting $\varphi_k^\J=P_{r(k)}^\J\varphi_k$, due to \eqref{e:sumisidentity}, we have
\begin{equation} \label{e:eqeigenfunctions}
\varphi_k=\varphi_k^\varnothing+\sum_{\J\in\mathcal{P}\setminus\{\varnothing\}} \varphi_k^\J.
\end{equation}
Let us prove that \eqref{e:eqeigenfunctions} implies \eqref{e:eqdefectmeasure}.  For that, we first recall an elementary lemma concerning the microlocal defect measure of a sum of sequences. It is proved in the case $p=2$ in \cite[Proposition 3.3]{gerard1991mesures} and a direct induction gives the general case.

\begin{lemma} \label{l:joint}
Let $p\in\N^*$ and $(u_k^1)_{k\in\N}, (u_k^2)_{k\in\N},\ldots, (u_k^p)_{k\in\N}$ be sequences of functions weakly converging to $0$, each with a unique microlocal defect measure $\mu_{1},\ldots,\mu_{p}$, respectively. We assume that $\mu_{1},\ldots,\mu_{p}$ are pairwise mutually singular. Then the sequence $(u_k^1+\cdots+u_k^p)_{k\in\N}$ has a unique microlocal defect measure, which is $\mu_1+\cdots+\mu_p$.
\end{lemma}

Combining Lemma \ref{l:supportnuJemptyset}, Lemma \ref{l:joint} and \eqref{e:eqeigenfunctions},  we obtain \eqref{e:eqdefectmeasure}, which finishes the proof of Proposition \ref{t:splittinggeneral}. 

\section{Preliminaries for the proof of Theorem \ref{t:mainQL}} \label{s:prelimmainQL}

This section is devoted to preliminary steps for the proof of Theorem \ref{t:mainQL}. We fix $m\geq 2$ and $\Delta_{\sR}=\Delta$ as in Section \ref{s:products}. Recall that the case $m=1$ has been handled in \cite[Proposition~3.2]{de2018spectral}.

\subsection{Reduction to a fixed $\mathcal{J}\in\mathcal{P}\setminus \{\varnothing\}$.}  The first step in the proof consists in reducing the analysis to the part of the QL above $\Sigma_\J$ for some $\J\in\myP\setminus\{\varnothing\}$, and it is achieved thanks to Proposition~\ref{t:splittinggeneral}. Thanks to Proposition \ref{t:splittinggeneral}, it is possible to assume that $(\varphi_k)_{k\in\N^*}$ is a sequence of eigenfunctions with eigenvalue tending to $+\infty$, and with a unique microlocal defect measure $\nu$, which can be assumed to be supported in $S\Sigma$. Indeed, thanks to Proposition \ref{t:splittinggeneral}, we can even assume that all the mass of $\nu$ is contained in $S\Sigma_\J$ for some $\mathcal{J}\in\mathcal{P}\setminus \{\varnothing\}$, i.e., $\nu=\nu^\J$ (simply by considering only the term $\varphi_k^{\mathcal{J}}$). Once we have established the decomposition
\begin{equation} \label{e:decdQJ}
\nu^\J=\int_{\mathbf{S}_\J}\nu^\J_s dQ^\J(s),
\end{equation}
Theorem \ref{t:mainQL} follows by just gluing all pieces of $\nu$ together thanks to Proposition \ref{t:splittinggeneral}.

Therefore, in order to establish Theorem \ref{t:mainQL}, we assume that the unique microlocal defect measure of $(\varphi_k)_{k\in\N^*}$ has no mass outside $S\Sigma_\J$ for some $\mathcal{J}\in\mathcal{P}\setminus \{\varnothing\}$. Due to the analysis done in Section \ref{s:proofofproposition}, there exists a function $r(k)$ tending to $+\infty$ as $k\rightarrow+\infty$ such that $\varphi_k^\J=P_{r(k)}^\J\varphi_k$ has the same microlocal defect measure as $\varphi_k$. Thus, to analyze this microlocal defect measure, we can replace without loss of generality $\varphi_k$ by $P_{r(k)}^\J\varphi_k$ which is still an eigenfunction with same eigenvalue. The new $\varphi_k$ satisfies \eqref{e:eqwithP} below. By symmetry, we can also assume that $\mathcal{J}=\{1,\ldots,J\}$ with $J=\text{Card}(\mathcal{J})$. 

To sum up, the sequence $(\varphi_k)_{k\in\N^*}$ that we consider is no more a general sequence of normalized eigenfunctions with eigenvalues tending to $+\infty$, but it satisfies the following property:
\begin{property} \label{p:property}
The sequence $(\varphi_k)_{k\in\N^*}$ is a bounded sequence of eigenfunctions of $-\Delta$ labeled with increasing eigenvalues tending to $+\infty$, and with unique microlocal defect measure $\nu$. Moreover, there exist $J\leq m$ and $r(k)\rightarrow +\infty$ as $k\rightarrow +\infty$ such that
\begin{equation} \label{e:eqwithP}
\varphi_k=P^{\mathcal{J}}_{r(k)}\varphi_k
\end{equation}
for $\mathcal{J}=\{1,\ldots,J\}$ and for any $k\in\N^*$, where $P_n^\J$ is defined in \eqref{e:defPn}. In particular, $\nu$ has no mass outside $S\Sigma_\J$.
\end{property}

\subsection{Illustration and sketch of proof}
Since the rest of the proof is slightly involved, in this section we provide an illustration and a sketch of proof. The proof is written in full details in Section \ref{s:prooftheo}. Logically, one may omit the discussion which follows and proceed directly to the next section.

\paragraph{Illustration of Theorem \ref{t:mainQL}.} To get an intuition of Theorem \ref{t:mainQL}, fix $(n_1,\ldots,n_m)\in\N^m$, and consider a sequence of normalized eigenfunctions $(\psi_k)_{k\in\N^*}$ of $-\Delta$ given in a tensor form as in Remark \ref{r:tensoreigen}, such that, for any $k\in\N^*$, $\psi_k$ is also, for any $1\leq j\leq m$, a sequence of eigenfunctions of $R_j$ with eigenvalue tending to $+\infty$, and of $\Omega_j$ with eigenvalue ${2n_j+1}$. We notice that any associated QL $\nu$ is supported in $S\Sigma$: it follows directly from the arguments developed in the proof of Proposition \ref{t:splittinggeneral}, since for any $1\leq j\leq m$, the eigenvalues with respect to $R_j^2$ (which plays the role of $Z_j^*Z_j$ in Assumption \ref{a:a}) are much larger than the eigenvalues with respect to $-\Delta$.

Let $\J=\{1,\ldots,m\}\in\mathcal{P}$. Then, $\nu$ is necessarily invariant under the vector field $\vec\rho^\J_s$, where $s=(s_1,\ldots,s_m)\in\mathbf{S}_\J$ is defined by $s_j=\frac{2n_j+1}{2n_1+1+\cdots+2n_m+1}$ for $j=1,\ldots,m$. To see it, we set 
\begin{equation*}
R=\frac{\sum_{j=1}^m (2n_j+1)R_j}{\sum_{j=1}^m 2n_j+1}
\end{equation*}
and we note that for any $A\in \Psi^0(\M)$, we have 
\begin{equation*}
([A,R]\psi_k,\psi_k)=(AR\psi_k,\psi_k)-(A\psi_k,R\psi_k)=0
\end{equation*}
since $\psi_k$ is an eigenfunction of $R$. In the limit $k\rightarrow +\infty$, taking the principal symbol, we obtain 
\[ \int_{S\Sigma} (\rho_s^\J a)d\nu=0, \] 
where $a=\sigma_P(A)$. Since it is true for any $a\in\mathscr{S}^0(\M)$ (the set of symbols of order $0$, see Appendix \ref{a:pseudo} for notations), this implies $\vec\rho^\J_s\nu=0$. Hence, for such sequences $(\psi_k)_{k\in\N^*}$, any QL $\nu$ is invariant under $\vec\rho_s^\J$ and $Q^\J$ is a Dirac mass on $s$ in the decomposition \eqref{e:decdQJ}.

Roughly speaking, any QL supported on $S\Sigma$ is a linear combination of sequences as in the above example, for different $\J\in\mathcal{P}\setminus \{\varnothing\}$ and different $s\in\mathbf{S}_\J$.

\paragraph{Roles of $R_j$ and $\Omega_j$.} The operators $R_j$ and $\Omega_j$ play a key role in the proofs of Theorem \ref{t:mainQL}, \ref{t:conversereplacement} and \ref{t:goodconversereplacement}. As illustrated in the previous paragraph, the operators $\Omega_j$ are linked with the parameters $s\in \mathbf{S}_\J$: in some sense, once the eigenfunctions have been orthogonally decomposed with respect to the operators $R_j$ and $\Omega_j$ (as explained in Section \ref{s:spec2}), the ratios between the $\Omega_j$-s determine the invariance property of the associated QLs through the parameter $s$ and the vector field $\vec\rho_s^\J$. On the other side, the operators $R_j$ `determine' the microlocal support of the associated QLs, for example they determine the element $\J\in\mathcal{P}\setminus\{\varnothing\}$ for which the QL concentrates on $S\Sigma_\J$.

\paragraph{Sketch of proof.} In order to simplify the presentation, in this sketch of proof, we assume that $\mathcal{J}=\{1,\ldots,m\}$ and we omit the use of subscripts involving $\mathcal{J}$, but the ideas are similar for any $\mathcal{J}\in\mathcal{P}\setminus \{\varnothing\}$. 

We notice that \eqref{e:eqwithP} together with the fact that $\mathcal{J}=\{1,\ldots,m\}$ ensures that $\varphi_k$ has no zero Fourier modes along the $z_j$ variables for any $j\in\{1,\ldots,m\}$. Let us use the decomposition \eqref{e:decompolaplacienintro} to write each $\varphi_k$ as a sum of eigenfunctions of operators of the form $\sum_{j=1}^{m}\big(2n_j+1\big)R_j$ for some integers $n_1,\ldots,n_m$:
\begin{align} \label{e:decompoeigenintro}
&\qquad \qquad \qquad\varphi_k=\sum_{(n_1,\ldots,n_m)\in\mathbb{N}^m} \varphi_{k,(n_1,\ldots,n_m)}, \\
&\text{with} \ \  \Omega_j\varphi_{k,(n_1,\ldots,n_m)}=(2n_j+1)\varphi_{k,(n_1,\ldots,n_m)}, \quad \forall \; 1\leq j\leq m. \nonumber
\end{align}

We will see in Section \ref{s:spec2} that the decomposition \eqref{e:decompoeigenintro} is orthogonal, and therefore each eigenfunction $\varphi_{k,(n_1,\ldots,n_m)}$ has the same eigenvalue $\lambda_k$ as $\varphi_k$. Then, we do a careful analysis of this decomposition into modes, which, in the limit $k\rightarrow +\infty$, gives the disintegration $\nu=\int_{\mathbf{S}}\nu_sdQ(s)$. 

We take a partition of $\N^m$ into $2^N$ thin positive cones $C_\ell^N$ (with $0\leq \ell\leq 2^N-1$) with vertex $V=\big(-\frac12,\ldots,-\frac12\big)$ (see Figure \ref{f:1}), and we group the eigenfunctions $\varphi_{k,(n_1,\ldots,n_m)}$ with index $(n_1,\ldots,n_m)$ in the same cone $C_\ell^N$ into a single eigenfunction
\begin{equation*}
\varphi_{k,\ell}^N=\sum_{(n_1,\ldots,n_m)\in C_\ell^N} \varphi_{k,(n_1,\ldots,n_m)}
\end{equation*}
of $-\Delta$. Since the cones $C_\ell^N$ partition $\N^m$, we have
\begin{equation} \label{e:decompoeigencones}
\varphi_k=\sum_{\ell=0}^{2^N-1} \varphi_{k,\ell}^N
\end{equation}
for any $N\in\N^*$. 

Taking a microlocal defect measure $\nu_\ell^N$ in each sequence $(\varphi_{k,\ell}^N)_{k\in\N^*}$ and making $N\rightarrow +\infty$, we obtain from \eqref{e:decompoeigencones} the disintegration $\nu=\int_{\mathbf{S}}\nu_s dQ(s)$. This follows from the fact that for any $s=(s_1,\ldots,s_m)\in\mathbf{S}$, there exists a sequence of positive cones $C_{\ell(N)}^N$ in the partition degenerating as $N\rightarrow +\infty$ to the half-line with vertex $V$ and parametrized by $s$: for this, choose a sequence of cones $(C_{\ell(N)}^N)_{N\in\N}$ for which the indices $(n_{1,N},\ldots,n_{m,N})\in (\N^m)^\N$ satisfy
\[ \Big(\frac{2n_{1,N}+1}{2n_{1,N}+1+\cdots+2n_{1,N}+1},\ldots,\frac{2n_{1,N}+1}{2n_{1,N}+1+\cdots+2n_{1,N}+1}\Big)\underset{N\rightarrow +\infty}{\longrightarrow} (s_1,\ldots,s_m). \]

For this choice of cones $C_{\ell(N)}^N$, $dQ(s)$ accounts for the relative mass, in the limit $N\rightarrow +\infty$, of the eigenfunction $\varphi_{k,\ell(N)}^N$ in the sum \eqref{e:decompoeigencones}.

The invariance property $\vec{\rho}_s\nu_s=0$ can be seen from the fact that, for any large $N$ and any $\ell=\ell(N)$ satisfying $0\leq \ell\leq 2^N-1$, each eigenfunction $\varphi_{k,(n_1,\ldots,n_m)}$ with $(n_1,\ldots,n_m)\in C_{\ell(N)}^N$ is indeed an eigenfunction of the operator
\begin{equation*}
\sum_{i=1}^m \Big(\frac{2n_i+1}{2n_1+1+\cdots+2n_m+1}\Big)R_i
\end{equation*}
and thus a quasimode of $R_s=s_1R_1+\cdots+s_mR_m$ if $s=(s_1,\ldots,s_m)\in\mathbf{S}$ denotes the parameter of the limiting half-line (with vertex $V$) of the positive cones $\smash{C_{\ell(N)}^N}$ as $N\rightarrow +\infty$. Hence, $\smash{\varphi_{k,\ell}^N}$ is an approximate eigenfunction of $R_s$, from which it follows by a classical argument that $\nu_s$ is invariant under the vector field $\vec\rho_s$ of $\rho_s=(\sigma_P(R_s))_{|\Sigma}$.

\subsection{Spectral and symplectic preliminaries} \label{s:spec2} \label{s:symplecticHamilt}

In this section, we gather a few facts which will be used in the proof of Theorem \ref{t:mainQL}.

We use the notations introduced in Section \ref{s:products}, notably $R_j$, $\Omega_j$ for the operators defined through a Fourier expansion with respect to the $z_j$-variables, and satisfying \eqref{e:decompolaplacienintro}. 

\begin{lemma} \label{l:propofOmegaandR} The following properties hold:
\begin{enumerate}
\item The operator $\Omega_j$, seen as an operator on the $j$-th copy of $\mathbf{H}$, has eigenvalues $2n+1$, $n\in\mathbb{N}$.
\item $[\Omega_i,\Omega_j]\varphi=[R_i,R_j]\varphi=[\Omega_i,R_j]\varphi=0$ for any $i,j$ and any $\varphi$ whose $0$-th Fourier mode with respect to $z_i$ and $z_j$ vanishes.
\item The operators $R_j$ and $\Omega_j$ are pseudodifferential operators in any cone of $T^*\M$ whose intersection with some conic neighborhood of the set $\{p_{z_j}=0\}$ is reduced to $0$, in particular on $\Sigma_\J$.
\item The Hamiltonian vector field associated to the Hamiltonian $\sigma_P(R_j)=|p_{z_j}|$ is $\text{sgn}(p_{z_j})\partial_{z_j}$.
\item The Hamiltonian flow $\theta_j(\cdot)$ associated to $\sigma_P(\Omega_j)$ is stationary on $\Sigma_\J$ when $j\in\J$.
\end{enumerate}
\end{lemma}
\begin{proof}
Point 1 is proved in \cite[Section 3.1]{de2018spectral}. Point 2 follows from the definition of $\Omega_i, R_j$ in Section \ref{s:productsoftheheisenberggroup} (they are defined only on the direct sum of the eigenspaces corresponding to non-zero eigenvalues of the operators $\partial_{z_i}$ and $\partial_{z_j}$).

Point 3 follows from the fact that in any conic set $U\subset S^*\M$ which is the complement of a conic neighborhood of $\{p_{z_j}=0\}$, $|p_{z_j}|$ is infinitely differentiable. It is indeed an elliptic first-order classical symbol in $U$. The standard quantization of $|p_{z_j}|$ is $R_j$, which is an elliptic first-order pseudodifferential operators when acting on functions microlocalized in $U$. Then $\Omega_j=\Delta_j/R_j$ is also an elliptic first-order pseudodifferential operators when acting on functions microlocalized in $U$.

Point 4 then follows from a direct computation. 

For Point 5, we notice that
\[ \sigma_P(\Omega_j)=\frac{h_{X_j}^2+h_{Y_j}^2}{|h_{\partial_{z_j}}|} \]
in the cones where $\Omega_j$ is a pseudodifferential operator. Since $h_{X_j}=h_{Y_j}=0$ on $\Sigma$,  this Hamiltonian vector field vanishes on $\Sigma_\J$, and $\theta_j$ is stationary on $\Sigma_{\J}$.
\end{proof}

\section{Proof of Theorem \ref{t:mainQL}} \label{s:prooftheo}   \label{s:QEtheorem}
In this section, building upon Section \ref{s:prelimmainQL}, we prove Theorem \ref{t:mainQL}. In the sequel, the notation $(\cdot, \cdot)$ stands for the $L^2(\M)$ scalar product, and the associated norm is denoted by $\lVert \cdot\rVert_{L^2}$. Also, we recall that we assumed $\mathcal{J}=\{1,\ldots,J\}$.

\subsection{Positive cones} We consider the quadrant 
\begin{equation*}
\mathcal{C}=\Big\{(x_1,\ldots,x_J)\in\R^J \mid x_j\geq -\frac12 \text{ for any } 1\leq j\leq J \Big\}.
\end{equation*}
and we define
\[ V=\Big(-\frac12,\ldots,-\frac12\Big)\in\R^J. \]
A positive cone with vertex at $V$ is a subset $K$ of $\mathcal{C}\setminus\{V\}$ such that 
\[ W\in K\Rightarrow V+\lambda(W-V)\in K \] 
for any $\lambda>0$ and any $W\in\mathcal{C}$.
We notice that any positive cone $K$ with vertex at $V$ can be split into two non-empty positive cones $K_1, K_2$ with vertex at $V$: for this, we choose a half-space $H$ containing $V$ in its boundary and containing some of the points of $K$ but not all, and we set $K_1=K\cap H$ and $K_2=K\cap H^c$. We call this a ``bisection of $K$''.

We now define a sequence of partitions of $\mathcal{C}$ into positive cones with vertex at $V$. We first partition $\mathcal{C}$ into $2$ cones by bisection of $\mathcal{C}$. This gives a first partition of $\mathcal{C}$. Then we obtain a second partition of $\mathcal{C}$ by bisecting each of these two cones. And so on and so forth, refining our partition at each step by bisecting all cones of the previous partition. The $N$-th partition is made of $2^N$ positive cones with vertex at $V$.

Formalizing this, these positive cones $C_\ell^N\subset \mathcal{C}$, for $N\in\mathbb{N}^*$ and $0\leq \ell\leq 2^N-1$, satisfy the following properties (see Figure \ref{f:1} below):
\begin{enumerate}[(1)]
\item For any $N\in\mathbb{N}^*$ and any $0\leq \ell\leq 2^N-1$, $C_\ell^N$ is a positive cone with vertex at $V$, i.e.,
\begin{equation*}
V+\lambda (W-V)\in C_\ell^N, \ \ \ \forall \lambda>0,\ \forall W\in C_\ell^N;
\end{equation*} 
\item For any $N\in\mathbb{N}^*$, $(C_\ell^N)_{0\leq \ell\leq 2^N-1}$ is a partition of $\mathcal{C}$, i.e.,
\begin{equation*}
\bigcup_{\ell=0}^{2^N-1} C_\ell^N=\mathcal{C} \text{ \ \  and \ \ } C_\ell^N\cap C_{\ell'}^N=\varnothing, \ \forall \ell\neq \ell';
\end{equation*}
\item Each partition is a refinement of the preceding one: for any $N\geq 2$ and any $0\leq \ell\leq 2^N-1$, there exists a unique $0\leq \ell'\leq 2^{N-1}-1$ such that $C_\ell^N\subset C_{\ell'}^{N-1}$.
\end{enumerate}
We also impose that the aperture of the positive cones $C_0^N,\ldots,C_{2^N-1}^N$ tends uniformly to $0$ as $N\rightarrow +\infty$. To give a precise statement of this last assumption, we denote by $\mathscr{L}$ the set of half-lines issued from $V$ and contained in $\mathcal{C}$, and we note that $\mathscr{L}$ is parametrized by $s\in\mathbf{S}_\J$. We assume the following
\begin{enumerate}[(4)]
\item There exists $d:\N\rightarrow \R^+$ with $d\rightarrow 0$ as $N\rightarrow +\infty$, such that for any $N\in\N$, any $\ell\in \{0,\ldots,2^N-1\}$ and any $s',s''$ parametrizing lines in $C_\ell^N$, we have 
\begin{equation} \label{e:diffs}
\lVert s'-s''\rVert_{1}\leq d(N).
\end{equation}
As a consequence, for any $L\in\mathscr{L}$ parametrized by $s\in\mathbf{S}_\J$, there exists a subsequence $(C_{\ell(s,N)}^N)_{N\in\mathbb{N}^*}$ that converges to $L$ in the sense that
\[ \bigcap_{N\in\N} C_{\ell(s,N)}^N=L. \]
\end{enumerate}

\begin{figure}[ht] 
\includegraphics[width=7cm]{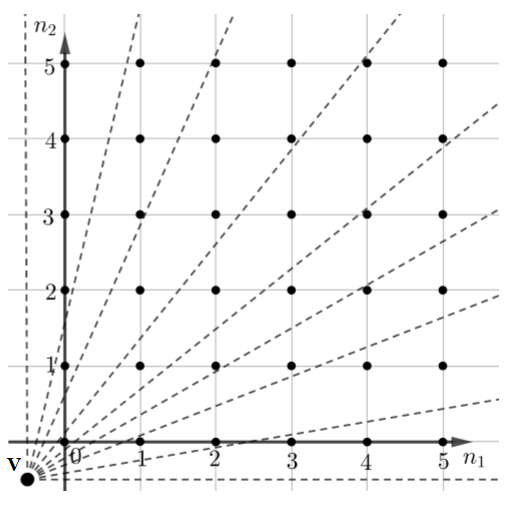}
\caption{The positive cones $C_\ell^N$, for $J=2$, $N=3$.}
\label{f:1}
\end{figure}

\begin{remark}
The positive cones $C_\ell^N$ can be seen as positive sub-cones of the Heisenberg fan (whose definition has been recalled in Section \ref{s:bibliocomments}).
\end{remark}

\subsection{Spectral decomposition of Quantum Limits}  Recall that we assumed $\mathcal{J}=\{1,\ldots,J\}$. We notice that \eqref{e:defPn} and \eqref{e:eqwithP} guarantee that $\varphi_k$ has no zero Fourier modes along the $z_j$ variables for any $j\in \J$. Using Point 2 of Lemma \ref{l:propofOmegaandR}, we can simultaneously diagonalize the operators $\Omega_j$ for $j\in\J$. This yields a decomposition of $\varphi_k$ on the joint eigenspaces of the $\Omega_j$ for $j\in\J$: according to \eqref{e:decompoeigenintro}, we obtain for any $(n_j)\in\N^\J$, $k\in\N^*$ and $j\in\J$ a function $\varphi_{k,(n_1,\ldots,n_J,0,\ldots,0)}$ such that 
\begin{equation*}
\Omega_j \varphi_{k,(n_1,\ldots,n_J,0,\ldots,0)}=(2n_j+1)\varphi_{k,(n_1,\ldots,n_J,0,\ldots,0)}.
\end{equation*}
Moreover we have
\begin{equation} \label{e:orthodecompophik}
\varphi_{k}=\sum_{\ell=0}^{2^N-1} \varphi_{k,\ell}^N
\end{equation}
where 
\begin{equation}\label{e:neededref}
\varphi_{k,\ell}^N=\sum_{(n_1,\ldots,n_J)\in C_{\ell}^N}\varphi_{k,(n_1,\ldots,n_J,0,\ldots,0)}.
\end{equation}
For any $N\in\mathbb{N}^*$ and any $0\leq \ell\leq 2^N-1$, we take 
\[ \nu_\ell^N \text{ a microlocal defect measure of the sequence } (\varphi_{k,\ell}^N)_{k\in\mathbb{N}^*}. \] 
By diagonal extraction in $k\in\N^*$ (which we omit in the notations), we can assume that any of these microlocal defect measures is obtained with respect to the same subsequence.

\begin{lemma}  \label{l:propnuellN}
The following properties hold:
\begin{enumerate}[(1)]
\item All the mass of $\nu_\ell^N$ is contained in $S\Sigma_\J$ for any $N\in\N^*$ and any $0\leq \ell\leq 2^N-1$;
\item  For $N\in\N^*$ and $\ell\neq \ell'$ with $0\leq \ell,\ell'\leq 2^N-1$, the joint microlocal defect measure (see Definition \ref{d:mdm}) of $(\varphi_{k,\ell}^N)_{k\in\mathbb{N}^*}$ and $(\varphi_{k,\ell'}^N)_{k\in\mathbb{N}^*}$ vanishes. In particular, for any $N\in\N^*$,
\begin{equation} \label{e:nuasasum}
\nu=\sum_{\ell=0}^{2^N-1} \nu_\ell^N.
\end{equation}
and also for any $N_0\leq N$ fixed and any $0\leq \ell_0\leq 2^{N_0}-1$,
\begin{equation} \label{e:nuellNasasum}
\nu_{\ell_0}^{N_0}=\sum_{\substack{\ell \text{ such that }\\ C_\ell^N\subset C_{\ell_0}^{N_0}}}\nu_\ell^N.
\end{equation}
\end{enumerate}
\end{lemma}
\begin{proof}
The proof mainly relies on averaging techniques (see also Appendix \ref{s:contact} for a result obtained by these techniques in the much simpler context of flat contact sub-Laplacians).

We first prove Point (1). Applying $P_{r(k)}^{\mathcal{J}}$ (see \eqref{e:eqwithP}) on both sides of \eqref{e:orthodecompophik}, we get that
\begin{equation*}
\sum_{\ell=0}^{2^N-1} P^{\mathcal{J}}_{r(k)}\varphi_{k,\ell}^N=P^{\mathcal{J}}_{r(k)}\varphi_{k}=\varphi_{k}=\sum_{\ell=0}^{2^N-1} \varphi_{k,\ell}^N.
\end{equation*}
We observe that $P_n^{\mathcal{J}}\in\Psi^0(\M)$ commutes with the operators $\Omega_j$ for $j\in\mathcal{J}$, thanks to its explicit expression \eqref{e:defPn}. Hence $P^{\mathcal{J}}_{r(k)}(\varphi_{k,(n_1,\ldots,n_m,0,\ldots,0)})=\varphi_{k,(n_1,\ldots,n_m,0,\ldots,0)}$ for any $(n_1,\ldots,n_m)\in\N^m$, and we deduce
\[ \varphi_{k,\ell}^N=P^{\mathcal{J}}_{r(k)}\varphi_{k,\ell}^N. \]
Point (1) now follows from the fact that $\sigma_P(P^{\mathcal{J}}_{r(k)})\rightarrow \mathbf{1}_{\Sigma_\J}$ as $k\rightarrow +\infty$ (see Lemma \ref{l:propofPnJ}).

We now turn to the proof of Point (2). 
 
Let $B\in\Psi^0(\M)$ be microlocally supported in a conic set in which $R_j, \Omega_j$ act as first-order pseudodifferential operators for any $j\in\J$. A typical example of microlocal support for $B$ is given by any conic subset of $T^*\M$ whose intersection with some conic neighborhood of the set $\{p_{z_j}=0\}$ is reduced to $0$, for any $j\in\J$. We set $U(t)=U(t_1,\ldots,t_J)=e^{i(t_{1}\Omega_{1}+\cdots+t_J\Omega_J)}$ for $t=(t_1,\ldots,t_J)\in(\R/2\pi\Z)^J$. 

The average of $B$ is then defined by (see \cite{weinstein1977asymptotics})
\begin{equation} \label{e:averageofb}
A=\int_{(\R/2\pi\Z)^J}U(-t)BU(t)dt.
\end{equation}
\begin{fact}\label{f:average}
There holds $[A,\Omega_j]=0$ for any $1\leq j\leq J$. Also, $\sigma_P(A)=\sigma_P(B)$ on $S\Sigma_\J$.
\end{fact} 
We postpone the proof of this fact to the end of the present section.
 
Let $N,\ell,\ell'$ be as in the statement of Point (2). The joint microlocal defect measure of $(\varphi_{k,\ell}^N)_{k\in\mathbb{N}^*}$ and $(\varphi_{k,\ell'}^N)_{k\in\mathbb{N}^*}$ has no mass outside $S\Sigma_\J$ (due to the fact that $\varphi_{k,\ell}^N=P^\J_{r(k)}\varphi_{k,\ell}^N$).  This, combined with the second part of Fact \ref{f:average}, yields 
 \begin{equation}\label{e:limittends0}
 (B\varphi_{k,\ell}^N,\varphi_{k,\ell'}^N)- (A\varphi_{k,\ell}^N,\varphi_{k,\ell'}^N) \underset{k\rightarrow +\infty}{\longrightarrow} 0.
 \end{equation}
 \begin{fact}  \label{r:ortho}
Let $D\in\Psi^0(\M)$ satisfy $[D,\Omega_j]=0$ for any $j\in \mathcal{J}$. Let $f,g$ be each in a joint eigenspace of the $\Omega_j$, meaning that for any $j\in\J$, $\Omega_jf=n_jf$, $\Omega_jg=n_j'g$ for some $n_j,n_j'\in 2\N+1$. We assume that there exists $j\in\J$ such that $n_j\neq n_j'$. Then $(Df,g)=0$ [because $D$ leaves any joint eigenspace of the $\Omega_j$ ($j\in\J$) invariant, and the eigenspaces are orthogonal].
\end{fact}

Since $A$ commutes with $\Omega_j$ for any $1\leq j\leq J$, by \eqref{e:neededref} and Fact \ref{r:ortho}, we know that $(A\varphi_{k,\ell}^N,\varphi_{k,\ell'}^N) =0$. Hence, plugging into \eqref{e:limittends0}, we get that $ (B\varphi_{k,\ell}^N,\varphi_{k,\ell'}^N)$ tends to $0$ as $k\rightarrow +\infty$. Using this result for all possible $B\in\Psi^0(M)$ with microlocal support satisfying the property recalled at the beginning of the proof, we obtain that the joint microlocal defect measure of $(\varphi_{k,\ell}^N)_{k\in\N^*}$ and of $(\varphi_{k,\ell'}^N)_{k\in\N^*}$ vanishes. Evaluating $(B\varphi_k,\varphi_k)$ in the limit $k\rightarrow +\infty$ and using \eqref{e:orthodecompophik}, we conclude the proof of Point (2).
\end{proof}
\begin{proof}[Proof of Fact \ref{f:average}]
For $1\leq j\leq J$, since 
\begin{equation*}
\frac{d}{dt_{j}}U(-t)BU(t)=iU(-t)[B,\Omega_{j}]U(t),
\end{equation*}
integrating in the $t_j$ variable, using that $\Omega_j$ commutes with $U(t)$, and that $\exp(2i\pi\Omega_j)=\Id$ (since the eigenvalues of $\Omega_j$ belong to $\N$),  we get that $[A,\Omega_j]=0$ for any $1\leq j\leq J$. 

For $1\leq j\leq J$, recall that $\theta_j(\cdot)$ denotes the flow of the Hamiltonian vector field of $\sigma_P(\Omega_j)$. By Egorov's theorem, $A$ has principal symbol
\begin{equation} \label{e:EgorovforA}
a:=\sigma_P(A)=\int_{(\R/2\pi\Z)^J}\sigma_P(B)\circ \theta_1(t_1) \circ \cdots \circ \theta_J(t_J)  \ dt
\end{equation}
(see \cite[Lemma 6.1]{de2018spectral} for similar arguments). Since $\theta_j$ is stationary on $\Sigma_\J$ for $1\leq j\leq J$ (see Lemma \ref{l:propofOmegaandR}), we get that $\sigma_P(A)=\sigma_P(B)$ on $S\Sigma_\J$.
\end{proof}

\subsection{Disintegration of measures} From the equality \eqref{e:nuasasum} taken in the limit $N\rightarrow +\infty$, we will deduce in this section that 
\[ \nu^\J=\int_{\mathbf{S}_\J} \nu^\J_{s}dQ^\J(s). \] 
Note that a simple Fubini argument does not suffice since $Q^\J$ is not the Lebesgue measure in general (it may contain Dirac masses). Instead, we have to adapt the proof of the classical disintegration of measure theorem (see \cite{rohlin1962fundamental}).

First of all, we define a measure $Q^\J$ over the simplex $\mathbf{S}_\J$ as follows. It has been explained at the beginning of Section \ref{s:prooftheo} that the set $\mathscr{L}$  of half-lines issued from $V$ and contained in $\mathcal{C}$ is parametrized by $s\in\mathbf{S}_\J$. For $N\in\mathbb{N}^*$ and $0\leq \ell\leq 2^N-1$, we consider the subset of $\mathbf{S}_\J$ given by
\begin{equation} \label{e:defClN}
\mathbf{S}_\ell^N=\big\{ s\in\mathbf{S}_\J, \ s \text{ parametrizes a half-line of $\mathscr{L}$ contained in } C_\ell^N \big\}.
\end{equation} 
Then we define  
\begin{equation} \label{e:defQ}
Q^\J(\mathbf{S}_\ell^N)=\nu^N_\ell(S\Sigma)
\end{equation}
and we extend it by finite additivity and complementation to the algebra of subsets of $\mathbf{S}_\J$ generated by the $\mathbf{S}_\ell^N$ for $N\in\N$ and $\ell\in\{0,\ldots,2^N-1\}$. Due to \eqref{e:nuellNasasum}, $Q^\J$ is a sigma-additive function on this algebra. Therefore, by the Caratheodory (or Hahn-Kolmogorov) extension theorem, \eqref{e:defQ} defines a (unique) non-negative Radon measure $Q^\J$ on the sigma-algebra generated by the cones $C_\ell^N$, which consists of the Borel sets of $\mathbf{S}_\J$.

Given $N\geq 1$, $0\leq \ell\leq 2^N-1$ and a continuous function $f:S\Sigma_\J\rightarrow \R$, we set 
\begin{equation} \label{e:deffjm}
f_\ell^N=\frac{1}{\nu_\ell^N(S\Sigma_\J)}\int_{S\Sigma_\J}fd\nu_\ell^N
\end{equation}
if $\nu_\ell^N(S\Sigma_\J)\neq 0$, and $f_\ell^N=0$ otherwise.
\begin{proposition} \label{e:convergencemuj}
Given any continuous function $f:S\Sigma\rightarrow \R$, for $Q^\J$-almost all $s\in \mathbf{S}_\J$, there exists a real number $e(f)(s)$ such that 
\begin{equation*}
f_{\ell(s,N)}^N\underset{N\rightarrow +\infty}{\longrightarrow} e(f)(s),
\end{equation*}
where, for any $N\in\mathbb{N}^*$, $\ell(s,N)$ is the unique integer $0\leq \ell(s,N)\leq 2^N-1$ such that $s\in\mathbf{S}_{\ell(s,N)}^N$. \\ In the sequel, we call $\ell(s,N)$ the approximation at order $N$ of $s$.
\end{proposition}
We postpone the proof of Proposition \ref{e:convergencemuj} to Section \ref{e:convergencemuj}.

From \eqref{e:nuasasum} and \eqref{e:deffjm}, we infer that for any $N\geq 1$, 
\begin{equation*}
\int_{S\Sigma_\J}fd\nu^\J=\sum_{\ell=0}^{2^N-1}\int_{S\Sigma_\J}fd\nu_\ell^N=\sum_{\ell=0}^{2^N-1}f_\ell^N\nu_\ell^N(S\Sigma_\J),
\end{equation*}
and the dominated convergence theorem together with the definition of $Q^\J$ and Proposition \ref{e:convergencemuj} yield
\begin{equation} \label{e:domconv}
\int_{S\Sigma_\J}fd\nu^\J=\int_{\mathbf{S}_\J}e(f)(s)dQ^\J(s).
\end{equation}

We see that for a fixed $s\in \mathbf{S}_\J$,
\begin{equation*}
C^0(S\Sigma_\J,\R)\ni f\mapsto e(f)(s)\in \R
\end{equation*}
is a non-negative linear functional on $C^0(S\Sigma_\J,\R)$. By the Riesz-Markov theorem, there exists a unique Radon probability measure $\nu^\J_s$ on $S\Sigma_\J$ such that
\begin{equation} \label{e:rieszrep}
e(f)(s)=\int_{S\Sigma_\J}fd\nu^\J_s.
\end{equation}

Putting \eqref{e:domconv} and \eqref{e:rieszrep} together, we get 
\begin{equation*}
\int_{S\Sigma_\J}fd\nu^\J=\int_{\mathbf{S}_\J}\bigg(\int_{S\Sigma_\J}fd\nu^\J_s\bigg)dQ^\J(s)
\end{equation*}
which is the desired disintegration of measures formula.

\subsection{Invariance of the measures $\nu_s^\J$} There remains to show that $\nu^\J_s$ is invariant under the flow generated by $\vec{\rho}^\J_s$. We start with an ``approximate invariance'' lemma.
\begin{lemma} \label{l:approxinvar}
Let $A$ be a $0$-th order pseudodifferential operator microlocally supported in a conic set where $R_j,\Omega_j$ act as first-order pseudodifferential operators for any $j\in\J$. Then there exists $C_A>0$ such that for any $N\in\mathbb{N}^*$, any $0\leq \ell\leq 2^N-1$ and any $s\in \mathbf{S}_\J$ such that the half-line issued from $V$ and defined by the $J$ equations $\frac{2x_j+1}{2x_1+1+\cdots+2x_J+1}=s_j$ (and $x_j\geq -1/2$) lies in $C_{\ell}^N$, there holds
\begin{equation} \label{e:ineqa}
\bigg|\int_{S\Sigma_\J} (\vec{\rho}_{s}^\J a) d\nu_{\ell}^N\bigg|\leq C_Ad(N) \nu_\ell^N(S\Sigma_\J).
\end{equation}
where $a=\sigma_P(A)$.
\end{lemma}
\begin{proof}
For the moment, we assume in addition to the assumptions of the statement that $A$ commutes with $\Omega_1,\ldots,\Omega_J$ and with $\Delta_{j}=X_j^2+Y_j^2$ for any $J+1\leq j\leq m$. The fact that it is sufficient to consider such $A$ will be justified later in the proof. Recall that $R_s$ has been defined in \eqref{e:defRs}.\\
 Using that $[A,R_{s}]$ commutes with $\Omega_1, \ldots, \Omega_J$ in order to kill crossed terms (see Fact \ref{r:ortho}), we have
\begin{align}
&( [A,R_{s}]\varphi_{k,\ell}^N,\varphi_{k,\ell}^N)\smallskip\nonumber\\
&\quad= \bigg( [A,R_{s}]\underset{(n_1,\ldots,n_J)\in C_{\ell}^N}{\sum} \varphi_{k,(n_{1},\ldots,n_J,0,\ldots,0)},\underset{(n_1,\ldots,n_J)\in C_{\ell}^N}{\sum} \varphi_{k,(n_{1},\ldots,n_J,0,\ldots,0)} \bigg) \nonumber \\
&\quad=\underset{(n_1,\ldots,n_J)\in C_{\ell}^N}{\sum} ( [A,R_{s}]\varphi_{k,(n_{1},\ldots,n_J,0,\ldots,0)}, \varphi_{k,(n_{1},\ldots,n_J,0,\ldots,0)}). \label{e:befcomp}
\end{align}
Let us fix $(n_1,\ldots,n_J)\in C_{\ell}^N$.  For simplicity of notations, we set $\varphi=\varphi_{k,(n_{1},\ldots,n_J,0,\ldots,0)}$. We prove that
\begin{equation}
([A,R_{s}]\varphi, \varphi )=\sum_{j=1}^J\Big( s_j-\frac{2n_{j}+1}{\sum_{i=1}^J 2n_{i}+1}\Big) ( [A,R_j] \varphi, \varphi) \label{e:aftcomp}
\end{equation}
We set 
\begin{equation*}
R=\frac{\sum_{j=1}^J (2n_{j}+1)R_j-\sum_{i=J+1}^m \Delta_{i}}{\sum_{j=1}^J 2n_{j}+1},
\end{equation*}
which, up to a constant, is the restriction of $-\Delta$ to the joint eigenspace of the $\Omega_j$ with eigenvalues $2n_j+1$.
 Using that $R$ is selfadjoint (since $R_j$ is selfadjoint for any $j$) and that $\varphi$ is an eigenfunction of $R$, we get
\begin{equation*}
( [A,R]\varphi, \varphi)= (AR\varphi,\varphi) - ( A\varphi, R\varphi) =0
\end{equation*}
and therefore, since $A$ commutes with $\Delta_{J+1},\ldots,\Delta_m$, we get
\begin{equation*}
( [A,R_{s}]\varphi,\varphi)=( [A,R_{s}-R]\varphi,\varphi)=\sum_{j=1}^J\Big( s_j-\frac{2n_{j}+1}{\sum_{i=1}^J 2n_{i}+1}\Big) ( [A,R_j] \varphi,\varphi)
\end{equation*} 
which is exactly \eqref{e:aftcomp}.

Thanks to our choice of microlocal support for $A$, we know that $[A,R_j]\in \Psi^0(\M)$ for $1\leq j\leq J$, and thus is it bounded in $L^2(\M)$. Combining \eqref{e:befcomp} and \eqref{e:aftcomp}, we obtain the existence of a constant $C_A>0$ depending on $A$ such that
\begin{equation} \label{e:ineqsymbol}
\begin{aligned}
\big| ( [A,R_{s}]\varphi_{k,\ell}^N,\varphi_{k,\ell}^N) \big|& \leq C_A\underset{(n_1,\ldots,n_J)\in C_{\ell}^N}{\sum} \sum_{j=1}^J\Big| s_j-\frac{2n_{j}+1}{\sum_{i=1}^J 2n_{i}+1}\Big|\, \lVert\varphi_{k,(n_{1},\ldots,n_J,0,\ldots,0)} \rVert_{L^2}^{2} \\
&\leq C_Ad(N) \lVert\varphi_{k,\ell}^N\rVert_{L^2}^{2}
\end{aligned}
\end{equation}
where in the last line, we used \eqref{e:diffs} and the fact that distinct joint eigenspaces of the $\Omega_j$ are orthogonal.

In order to pass to the limit $k\rightarrow +\infty$ in these last inequalities, we use the following lemma.
\begin{lemma}\label{l:principalsymbol}
On $\Sigma_{\mathcal{J}}$, there holds
\begin{equation*}
\sigma_{P}([A,R_{s}])_{|\Sigma_{\mathcal{J}}}= i\vec{\rho}_{s}^\J a
\end{equation*}
where $a=\sigma_P(A)$.
\end{lemma}
\begin{proof}[Proof of Lemma \ref{l:principalsymbol}]
Denoting by $\{\cdot,\cdot\}$ the Poisson bracket on the symplectic manifold $T^*\M$, we have for $(q,p)\in\Sigma_\J$
\begin{align*}
\sigma_{P}([A,R_{s}])(q,p)=\frac1i\{a,\sum_{j\in\J}\text{sgn}(p_{z_j})s_jh_{\partial_{z_j}}\}(q,p)&=i\sum_{j\in\J}\text{sgn}(p_{z_j})s_j(\vec{h}_{\partial_{z_j}}a)(q,p)\\&=i(\vec{\rho}_s^\J a)(q,p)
\end{align*}
where in the last equality we used \eqref{e:vecrhoJs}. 
\end{proof}
 Since all the mass of $\nu_\ell^N$ is contained in $S\Sigma_\J$ by Lemma \ref{l:propnuellN}, we finally deduce the upper bound \eqref{e:ineqa} from \eqref{e:ineqsymbol} and Lemma \ref{l:principalsymbol}.

We have established this upper bound only for an operator $A$ of order $0$ \emph{which commutes with $\Omega_1,\ldots,\Omega_J$ and $\Delta_j$ for any $J+1\leq j\leq m$}. We would now like to remove this commutation assumption. 

Let $b\in\mathscr{S}^0(\M)$ be an arbitrary $0$-th order symbol supported in a subset of $T^*\mathbf{H}^{\mathcal{J}}$ where $R_j,\Omega_j$ act as first-order pseudodifferential operators for any $j\in\J$. Let also $(q,p)$ denote the coordinates in $T^*\M$ and $(q_j,p_j)$ the coordinates in the cotangent bundle of the $j$-th copy of $\mathbf{H}$.

We notice that $S\Sigma_\J$ is invariant under translation in $q_i$ for $i\notin \J$. Hence, considering the averaged symbol 
\begin{align*}
&\bar{b}(q_1,\ldots,q_m,p_1,\ldots,p_m)\nonumber\\
&\qquad =\frac{1}{|\text{Vol}(\mathbf{H}^{m-J})|}\int_{\mathbf{H}^{m-J}} b(q_1,\ldots,q_m,p_1,\ldots,p_J,0,\ldots,0)\; dq_{J+1}\ldots dq_m, \label{e:averagebb}
\end{align*}
  we have the following properties:
\begin{enumerate}[(i)]
\item $\bar{b}\in\mathscr{S}^0(\M)$ does not depend on $q_i,p_i$ for $J+1\leq i\leq m$;
\item for any vector field $X$ on $S\Sigma_\J$ depending only on the coordinates $q_1$, $\ldots$, $q_J$, $p_1$, $\ldots$, $p_J$, there holds
\begin{equation}\label{e:Xbarb}
\int_{S\Sigma_\J} (X\bar{b})d\nu_\ell^N=\int_{S\Sigma_\J} (Xb)d\nu_\ell^N.
\end{equation}
\end{enumerate}
We denote by $\Op^{\text{st}}$ the standard quantization (see Appendix \ref{a:pseudo}). We set
\begin{equation}\label{e:Aaveraged}
 A=\int_{(\R/2\pi\Z)^{J}}U(-t)\Op^{\text{st}}(\bar{b})U(t)dt \ \in \Psi^0(\M)
\end{equation}
where $U(t)=U(t_1,\ldots,t_J)=e^{i(t_{1}\Omega_{1}+\cdots+t_J\Omega_J)}$ for $t=(t_{1},\ldots,t_J)\in(\R/2\pi\Z)^{J}$. 
\begin{itemize}
\item $A$ commutes with $\Omega_j$ for any $1\leq j\leq J$. This follows from an argument that we have already described in the proof of Point (2) of Lemma \ref{l:propnuellN}, applied to  the formula \eqref{e:Aaveraged}.
\item $A$ also commutes with $\Delta_j$ for any $J+1\leq j\leq m$ thanks to (i), combined with the fact that the standard quantization preserves the product structure of the manifold $\M$ (see after \eqref{e:defstandardquantization}).
\end{itemize}

The principal symbol of $A$ on $S\Sigma_\J$ coincides with $\bar{b}$, due to \eqref{e:EgorovforA} and the fact that $\theta_j$ is stationary on $\Sigma_\J$ for $j\in\J$. Using \eqref{e:ineqa} for $A$, this proves that
\begin{equation*}
\bigg|\int_{S\Sigma_\J} \vec{\rho}_{s}^\J\bar{b}\ d\nu_{\ell}^N\bigg|\leq C_Ad(N) \nu_\ell^N(S\Sigma_\J).
\end{equation*}
But thanks to \eqref{e:Xbarb} applied with $X=\vec{\rho}_s^\J$ (recall that by formula \eqref{e:vecrhoJs}, $\vec{\rho}_s^\J$ does not depend on $q_i$ for $J+1\leq i\leq m$) there holds
\begin{equation*}
\int_{S\Sigma_\J} (\vec{\rho}_{s}^\J\bar{b}) d\nu_{\ell}^N=\int_{S\Sigma_\J} (\vec{\rho}_{s}^\J b) d\nu_{\ell}^N
\end{equation*}
hence the conclusion.
\end{proof}

We finally show how to deduce from Lemma \ref{l:approxinvar}  that $\nu^\J_s$ is invariant under the flow $e^{t\vec\rho_s^\J}$. Let $A\in \Psi^0(\M)$ be microlocally supported in a cone of $T^*\M$ whose intersection with some conic neighborhood of the set $\{p_{z_j}=0\}$ is reduced to $0$, for any $j\in\J$. We set $a=\sigma_P(A)$. For $Q^\J$-almost every $s\in \mathbf{S}_\J$, we have
\begin{align}
\bigg|\int_{S\Sigma_\J}\left(\vec{\rho}_s^\J a\right)d\nu^\J_s\bigg|&=\big|e(\vec{\rho}_s^\J a)(s)\big| \qquad \text{(by \eqref{e:rieszrep})} \nonumber \\
&=\bigg|\lim_{N\rightarrow +\infty} \frac{1}{\nu_{\ell(s,N)}^N(S\Sigma_\J)}\int_{S\Sigma_\J}(\vec{\rho}_s^\J a) d\nu_{\ell(s,N)}^N \label{e:dennull}\bigg| \\
&\leq \lim_{N\rightarrow +\infty} C_Ad(N) \qquad \text{(by \eqref{e:ineqa})}\nonumber \\
&=0 \nonumber
\end{align}
with the convention that if the denominator in \eqref{e:dennull} is null, then the whole expression is null. Then we conclude the proof by applying the following fact to $\vec{\rho}_s^\J$ which is a vector field on $S\Sigma_\J$:\\
\textbf{Fact.} Let $W$ be a manifold, equipped with a measure $\delta$, and let $T$ be a complete vector field on $W$. If $\int_W (T\phi) d\delta=0$ for every $\phi\in C_c^\infty(W,\R)$, then the measure $\delta$ is invariant under the flow of $T$. This is proved by considering the derivative $\frac{d}{dt}\int_W \phi(e^{tT}w)d\delta(w)$ at $t=0$.

\subsection{Proof of Proposition \ref{e:convergencemuj}}
In this section, we finally prove Proposition \ref{e:convergencemuj}. By linearity of formula \eqref{e:deffjm}, it is sufficient to prove the statement for $f\geq 0$. Therefore, in the sequel, we fix $f\geq 0$. For $N\geq 1$, we define the function $f^N:\mathbf{S}_\J\rightarrow\R$ by $f^N(s)=f_{\ell(s,N)}^N$, where $\ell(s,N)$ is the approximation at order $N$ of $s$. Note that $f^N$ is constant on $\mathbf{S}_\ell^N$ for $0\leq \ell\leq 2^N-1$.

For $0\leq \alpha < \beta\leq 1$, we define $S(\alpha,\beta)$ as the set of $s\in \mathbf{S}_\J$ such that
\begin{equation*}
\liminf_{N\rightarrow +\infty} f^N(s) <\alpha <\beta <\limsup_{N\rightarrow +\infty} f^N(s).
\end{equation*}
To prove Proposition \ref{e:convergencemuj}, it is sufficient to prove that $S(\alpha,\beta)$ has $Q^\J$-measure $0$ for any $0\leq \alpha<\beta\leq 1$. Fix such $\alpha,\beta$. For $s\in S(\alpha,\beta)$, take a sequence $1\leq N_1^\alpha(s)<N_1^\beta(s)<N_2^\alpha(s)<N_2^\beta(s)<\cdots<N_k^\alpha(s)<N_k^\beta(s)<\cdots$ of integers such that $f^{N_k^\alpha(s)}(s)<\alpha$ and $f^{N_k^\beta(s)}(s)>\beta$ for any $k\geq 1$. We finally define the following sets:
\begin{equation*}
A_k=\bigcup_{s\in S(\alpha,\beta)} \mathbf{S}_{\ell(s,N_k^\alpha(s))}^{N_k^\alpha(s)}
\end{equation*}
\begin{equation*}
B_k=\bigcup_{s\in S(\alpha,\beta)}  \mathbf{S}_{\ell(s,N_k^\beta(s))}^{N_k^\beta(s)}
\end{equation*}
We have $S(\alpha,\beta)\subset A_{k+1}\subset B_k\subset A_k$ for every $k\geq 1$. In particular, 
\begin{equation} \label{e:inclS}
S(\alpha,\beta)\subset\widetilde{S}(\alpha,\beta):=\bigcap_{k\in\N^*} A_k = \bigcap_{k\in\N^*} B_k.
\end{equation}

Given any two of the sets $\mathbf{S}_{\ell(s,N_k^\alpha(s))}^{N_k^\alpha(s)}$ that form $A_k$, either they are disjoint or one is contained in the other. Consequently, $A_k$ may be written as a disjoint union of such sets, denoted by $A_k^{k'}$. This union is countable, since the number of sets $\mathbf{S}^N_\ell$ ($N\in\N^*, 0\leq \ell\leq 2^N-1$) is countable. Therefore, 
\begin{equation*}
\int_{A_k}fdQ^\J =\sum_{k'}\int_{A_k^{k'}} fdQ^\J<\sum_{k'}\alpha Q^\J(A_k^{k'})=\alpha Q^\J(A_k)
\end{equation*}
and analogously, with similar notations, 
\begin{equation*}
\int_{B_k}fdQ^\J =\sum_{k'}\int_{B_k^{k'}} fdQ^\J>\sum_{k'}\beta Q^\J(B_k^{k'})=\beta Q^\J(B_k).
\end{equation*}
Since $B_k \subset A_k$, we get $\alpha Q^\J(A_k)>\beta Q^\J(B_k)$. Taking the limit $k\rightarrow +\infty$, it yields 
\[ \alpha Q^\J(\widetilde{S}(\alpha,\beta))>\beta Q^\J(\widetilde{S}(\alpha,\beta)), \] 
which is possible only if $Q^\J(\widetilde{S}(\alpha,\beta))=0$. Therefore, using \eqref{e:inclS}, we get $Q^\J(S(\alpha,\beta))=0$, which concludes the proof of the proposition.

\section{Proof of Theorem \ref{t:conversereplacement}} \label{s:converse}

\subsection{Preliminary steps}

In this subsection, we introduce the tools used in the proof of Theorem \ref{t:conversereplacement}.

\subsubsection{Spectral decomposition}\label{s:spectralspectral}

We first introduce a spectral decomposition of $-\Delta$. Fix $j\in \{1,\ldots,m\}$, consider the $j$-th copy of $L^2(\mathbf{H})$ in $L^2(\M)\cong L^2(\mathbf{H})^{\otimes m}$, and take the Fourier decomposition with respect to $z_j$ in this copy:
\begin{equation*} 
L^2(\mathbf{H})=L^2_{0}\oplus\underset{n \in \mathbb{N}, \alpha\in\Z\setminus \{0\}}{\bigoplus} E_{n,\alpha}
\end{equation*}
where $\partial_{z_j}$ acts as $0$ on $L^2_{0}$, and on $E_{n,\alpha}$, $\frac1i\partial_{z_j}$ acts as $\alpha$ and $\Omega_j$ as $2n+1$. \\
Recall that $\mathcal{P}$ stands for the set of all subsets of $\{1,\ldots,m\}$. We fix $\mathcal{J}\in\mathcal{P}$. For $(n_j)\in\N^{\mathcal{J}}$, $(\alpha_j)\in(\Z\setminus\{0\})^\J$ we set
\begin{equation*}
\mathcal{H}^{\mathcal{J}}_{(n_j),(\alpha_j)}=F^1\otimes \cdots \otimes F^m\subset L^2(\M)
\end{equation*}
where $F^j=E_{n_j,\alpha_j}$ for $j\in\mathcal{J}$ and $F^j=L^2_{0}$ for $j\notin\J$.

We have the orthogonal decomposition 
\begin{equation} \label{e:decompoL2}
L^{2}(\M)=\bigoplus_{\J\in\mathcal{P}}\underset{\substack{(n_j) \in \mathbb{N}^{\mathcal{J}}\\(\alpha_j) \in (\Z\setminus\{0\})^{\mathcal{J}}}}{\bigoplus}\mathcal{H}_{(n_j),(\alpha_j)}^{\mathcal{J}}.
\end{equation}
We can also write the associated decomposition of $-\Delta$ as 
\[ 
-\Delta =\bigoplus_{\J\in\mathcal{P}}\underset{\substack{(n_j) \in \mathbb{N}^{\mathcal{J}}\\(\alpha_j) \in (\Z\setminus\{0\})^{\mathcal{J}}}}{\bigoplus} H_{(n_j),(\alpha_j)}^{\mathcal{J}} \]
with
\[ H_{(n_j),(\alpha_j)}^{\mathcal{J}}  =\sum_{j\in \mathcal{J}}\big( 2n_{j}+1\big)|\alpha_j|-\sum_{i \notin \mathcal{J}} (\partial_{x_i}^2+\partial_{y_i}^2). \]
From this, we deduce
\begin{align*}
&\text{spec}(-\Delta)\nonumber\\&=\bigg\{ \sum_{j\in \mathcal{J}} \big(2n_j+1\big)|\alpha_j|+2\pi\sum_{i\notin \mathcal{J}} (k_i^2+\ell_i^2), \text{  with } k_i,\ell_i\in \Z, \ \mathcal{J}\in\mathcal{P}, \ n_j\in\N, \ \alpha_j\in \Z\setminus\{0\}\bigg\} 
\end{align*}
where $\text{spec}$ denotes the spectrum.

\subsubsection{Notations.} \label{s:notationsthconverse}

In this section, we define two sets $\myB^\J$ and $\myC^\J$ of Radon probability measures, contained in $\D^\J$, and which can be seen as ``elementary building blocks'' for proving Theorem \ref{t:conversereplacement}. For this we introduce a few more notations, in addition to those introduced in Section \ref{s:mainresultssection}.

Fix $\J\in\mathcal{P}\setminus\{\varnothing\}$. We first define an equivalence relation $\overset{\J}{\sim}$ on points in $\mathbf{H}^\J$: two points $q,q'$ are in relation if they can be obtained from each other by following the flows of $\partial_{z_j}$, $j\in\J$, i.e., if there exists $(s_j)\in\R^\J$ such that
\begin{equation*}
q'=\exp\bigg(\sum_{j\in\J} s_{j}\partial_{z_j}\bigg)q.
\end{equation*}
Fix $q\in\mathbf{H}^\J$. From the group law \eqref{e:grouplaw}, we see that the equivalence class of $q$ is an embedded submanifold of $\mathbf{H}^\J$, homeomorphic to the torus $(\R/2\pi\Z)^\J$. There is a probability measure on $\mathbf{H}^\J$ which is a uniform Dirac delta measure supported on this torus. Tensorizing this measure with the Lebesgue measure in $\mathbf{H}^{\notin\J}$ we obtain a Radon probability measure on $\M$, which we denote by $\alpha_q$. Its support is denoted by $M_q^\J\subset \M$.

We now lift $\alpha_q$ to $S^*M$: we define a probability measure on $S^*\M$ supported on $S\Sigma_\J$ and whose pushforward under the canonical projection from $T^*\M$ to $\M$ is exactly $\alpha_q$. For that, we notice that if $(q,p)\in S\Sigma_\J$, for any $q'\in\M$ it makes sense to consider the point $(q',p)\in S\Sigma_\J$, which is the point in the fiber of $S\Sigma$ over $q$ that has the same homogeneous coordinates $[p_{z_1}:\cdots:p_{z_m}]$ as $p$.
This allows us to define 
\begin{equation*}
\mathscr{H}_{q,p}^\J=\alpha_q\otimes \delta_p
\end{equation*}
where $\delta_p$ denotes the Dirac mass on $p$.

In addition to $\D^\J$, which has been introduced in \eqref{s:mainresultssection}, we introduce two other sets: 
\begin{itemize}
\item The set
\begin{equation}\label{e:defbj}
\myB^\J=\big\{\mathscr{H}_{q,p}^\J \mid  (q,p) \in S\Sigma_\J \big\}
\end{equation}
which is a subset of $\D^\J$. This follows from the definition of $\D^\J$ and the fact that $\alpha_q$ is invariant under $\partial_{z_j}$ for any $j\in\J$.
\item The set of convex combinations of measures in $\myB^\J$, which consequently is also a subset of $\D^\J$:
\begin{equation*}
\myC^\J=\bigg\{\sum_{i\in\mathcal{F}} \beta_i\nu_i \mid \mathcal{F} \text{ is a finite set}, \sum_{i\in\mathcal{F}} \beta_i=1, \forall i\ \beta_i\geq 0, \nu_i\in\myB^\J\bigg\}.
\end{equation*}
\end{itemize}

\subsection{Core of the proof}
In this section, we provide a fully detailed proof of Theorem \ref{t:conversereplacement}.  Proofs with relatively similar ideas can be found in \cite{jakobson1996classical}, \cite{Mac08}, \cite{studnia2019}, \cite{arnaiz2022}. The proof uses two main ingredients:
\begin{itemize}
\item On the ``classical side,''  the knowledge of the flows of the vector fields $\vec\rho_s^\J$ given by \eqref{e:vecrhoJs} (see also Remark \ref{r:flowrho}).
\item On the ``quantum side,'' the specific algebraic structure of $\text{spec}(-\Delta)$ (see Section \ref{s:spectralspectral}). 
\end{itemize}

\subsubsection{Step 1: Homogeneity}\label{s:step0}
In this preliminary step, we describe the homogeneity properties of $\M$.

The manifold $\mathbf{H}$ has a Lie group structure recalled in Section \ref{s:productsoftheheisenberggroup}, and thus $\M$ also has a Lie group structure obtained by product, whose composition law is denoted by $\star_m$. The left-translation by $g$ is denoted by $\tau_g$: $\tau_gq=g\star_m q$. The vector fields $X_j$ and $Y_j$ are left-invariant for $\star_m$, and thus $\Delta$ is also left-invariant. This implies that the left-translation by $g\in\M$ of an eigenfunction $\varphi$ of $\Delta$ is also an eigenfunction with same eigenvalue, denoted by $\tau_g\varphi=\varphi\circ \tau_g^{-1}$.

\begin{lemma}\label{l:step0}
The QLs of a sequence of $L^2(\M)$-normalized eigenfunctions $(\tau_g\varphi_k)_{k\in\N^*}$ are the left-translates by $g$ of the QLs of the sequence $(\varphi_k)_{k\in\N^*}$.
\end{lemma}
\begin{proof}
The left-translation $\tau_g$ induces an action on $\Psi^0(\M)$. The image of $A\in\Psi^0(\M)$ under left-translation by $g$ is denoted by $\widetilde{\tau}_g A\in \Psi^0(\M)$: it is defined as $(\widetilde{\tau}_gA)(\varphi)(q)=A(\tau_g^{-1}\varphi)(\tau_g^{-1}q)$. We have
\begin{equation}\label{e:actionope}
((\widetilde{\tau}_gA)(\tau_g\varphi),\tau_g\varphi)_{L^2(\M,\mu)}=(A\varphi,\varphi)_{L^2(\M,\mu)}
\end{equation}
 using that the Haar measure $\mu$ is left-invariant. 
 
The left-translation $\tau_g$ also induces an action $\overline{\tau}_g$ on the cotangent bundle $T^*\M$. We have 
\begin{equation}\label{e:actioncota}
\sigma_P(\widetilde{\tau}_gA)=\sigma_P(A)\circ \overline{\tau}_g^{-1}.
\end{equation}
Combining \eqref{e:actionope}, \eqref{e:actioncota} and the definition of QLs, we get the result.
\end{proof}

\subsubsection{Step 2}
Our goal in this subsection is to prove:
\begin{lemma} \label{l:buildblock1} 
Any $\nu^\J\in\D^\J$ is a QL, associated to an $L^2(\M)$-normalized sequence of eigenfunctions which are also eigenfunctions of $\Omega_j$ for $j\in\J$ (with eigenvalue $1$).
\end{lemma} 
We loosely follow the scheme of proof (and some proofs) of \cite{studnia2019}. 

According to the Calder\'{o}n--Vaillancourt theorem, there exists $K\in\N$ such that for any $a\in\mathscr{S}^0(\M)$, there holds
\begin{equation}\label{e:cvai}
(\Op(a)u,u)\leq C\lVert a\rVert_{C^K(S^*\M)}\lVert u\rVert_{L^2}^2.
\end{equation}
To avoid any confusion, we call \emph{weak-$*_K$ topology}  the weak-* topology associated to testing against elements of $C^K(S^*\M)$, and keep the terminology weak-* topology for the one associated to testing against elements of $C^0(S^*\M)$.

We prove two facts, which combined together imply Lemma \ref{l:buildblock1}.
\begin{fact} \label{f:density}
$\myC^\J$ is dense in the weak-$*_K$ topology in $\D^\J$.
\end{fact}
\begin{fact}\label{l:simplebuildblock23}
Any $\nu\in\myC^\J$ is a QL, associated to an $L^2(\M)$-normalized sequence of eigenfunctions which are also eigenfunctions of $\Omega_j$ for $j\in\J$ (with eigenvalue $1$).
\end{fact}
To explain how to deduce Lemma \ref{l:buildblock1} from these two facts, we first need to metrize the weak-$*_K$ topology (a direct diagonal extraction argument using Facts \ref{f:density} and \ref{l:simplebuildblock23} is not sufficient to prove Lemma \ref{l:buildblock1}).

We denote by $\mathcal{X}$ the unit ball of the topological dual of $C^K(S^*\M)$.  When equipped with the weak-$*_K$ topology, $\mathcal{X}$ is denoted by $\mathcal{X}_{*_K}$.
We construct a metric $\delta$ on $\mathcal{X}$ defining the same topology as the weak-$*_K$ convergence.  We pick a countable sequence $(a_r)_{r\in\N}$ which is dense in $C^K(S^*\M)$ (this space is separable), and we define $\delta:\mathcal{X}\times \mathcal{X}\rightarrow \R$ by
\begin{equation*}
\delta(\ell,\ell')=\sum_{r\in\N}\min\big((\ell-\ell')(a_r),2^{-r}\big).
\end{equation*}
This is a metric on $\mathcal{X}$, and $\text{Id}:\mathcal{X}_{*_K}\rightarrow (\mathcal{X},\delta)$ is a continuous bijection, hence an homeomorphism (since $\mathcal{X}_{*_K}$ is compact by the Banach-Alaoglu theorem). We conclude that
\begin{equation}\label{e:metrictopology}
\text{the topology induced by $\delta$ coincides with the weak-$*_K$ topology on $\mathcal{X}$.}
\end{equation}
\begin{proof}[Proof of Lemma \ref{l:buildblock1}] We notice that both $\myC^\J$ and $\D^\J$ are included in $\mathcal{X}$. Let $\nu^\J\in\D^\J$. For $r\in \N$ let $\nu_r\in \myC^\J$ with $\nu_r\rightarrow \nu^\J$ in the weak-$*_K$ topology as $r\rightarrow +\infty$. Such a sequence exists thanks to Fact \ref{f:density}. 

For any $r\in \N$, let $\varphi_{r,n}$ be an $L^2$-normalized eigenfunction of $\Delta$ which is also an eigenfunction of $\Omega_j$ for $j\in\J$ with eigenvalue $1$, and such that $\nu_{r,n}\in \mathcal{X}$ defined by
\begin{equation*}
\forall a\in C^K(S^*\M), \qquad (\Op(a)\varphi_{r,n},\varphi_{r,n})=\nu_{r,n}(a)
\end{equation*}
verifies 
\begin{equation}\label{e:mech}
\forall a\in\mathscr{S}^0(\M), \qquad \nu_{r,n}(a)\underset{n\rightarrow +\infty}{\longrightarrow} \int_{S^*M}ad\nu_r.
\end{equation}
Such a sequence exists thanks to Fact \ref{l:simplebuildblock23}. And due to \eqref{e:cvai} the convergence \eqref{e:mech} can be extended to any $a\in C^K(S^*\M)$. Fix $\varepsilon>0$. Pick $r\in\N$ sufficiently large such that $\delta(\nu_r,\nu^\J)\leq \varepsilon$. For $n\in\N$ sufficiently large, $\delta(\nu_{r,n},\nu_r)\leq \varepsilon$. Thus $\delta(\nu_{r,n},\nu^\J)\leq 2\varepsilon$. Taking $\varepsilon\rightarrow 0$, we obtain a sequence of eigenfunctions $(\psi_k)_{k\in\N}$ of the form $\psi_k=\varphi_{r_k,n_k}$ such that
\begin{equation*}
(\Op(a)\psi_k,\psi_k)=\int_{S^*M}ad\nu_{r_k,n_k}=\int_{S^*M} ad\nu^\J+o(1)
\end{equation*}
for any $a\in \mathscr{S}^0(\M)\cap C^K(S^*\M)=\mathscr{S}^0(\M)$. It follows that $\nu^\J$ is a QL.
\end{proof}

\subsubsection{Proof of Fact \ref{f:density}}

For this proof of density, we argue in two steps, see \eqref{e:d1} and \eqref{e:d2} below. The reason why we cannot argue directly with the  Krein--Milman theorem is that $S\Sigma_\J$ being not closed, the set $\myC^\J$ is not compact. 

Although not closed, the set $S\Sigma_\J$ can be written as an increasing and countable union of compact sets $A_r^\J$, $r\in\N$, namely
\begin{equation*}
A_r^\J=\bigg\{(q,p)\in S\Sigma_\J ~\bigg|~ \frac{|p_{z_j}|}{\underset{k\in\J}{\sum} |p_{z_k}|}\geq 1/r \text{ for any } j\in\J\bigg\}.
\end{equation*}
For any $r\in\N$, we denote by $\D_{r}^\J$ the subset of $\D^\J$ (introduced in \eqref{e:defDJ}) containing the Radon probability measures which are supported in $A_r^\J$.

\begin{lemma}\label{l:convcomp}
The set $\mathscr{D}_r^\J$ is convex, and compact for the weak-$*_K$ topology.
\end{lemma}
\begin{proof}
The invariance property involved in the definition of $\D^\J$ (and hence of $\D_r^\J$) can be equivalently stated with a set of  equations involving only continuous functions on $S^*\M$, and moreover  $A_r^\J$ is closed, therefore $\D_r^\J$ is closed for the weak-* topology. Since $\mathscr{D}_r^\J$ contains only Radon measures and using that $C^K(S^*\M)$ is dense in $C^0(S^*\M)$, we deduce that $\mathscr{D}_r^\J$ is also closed for the weak-$*_K$ topology. Since $\mathcal{X}_{*_K}$ is compact, it follows that $\D_r^\J\subset \mathcal{X}$ is compact for the weak-$*_K$ topology.  

Finally, $\D_{r}^\J$ is convex due to the inclusion 
\begin{equation*}
\text{supp}(t\mu_1+(1-t)\mu_2)\subset \text{supp}(\mu_1)\cup \text{supp}(\mu_2)
\end{equation*}
valid for any measures $\mu_1,\mu_2$. In order to apply the Krein--Milman theorem, we prove the following lemma.
\end{proof}

\begin{lemma}\label{l:extremal}
Any extremal point of $\D_{r}^\J$ is in $\myB^\J$.
\end{lemma}
\begin{proof} Assume for the sake of a contradiction that $\theta$ is an extremal point of $\D_{r}^\J$ such that $\theta\notin\myB^\J$. According to \eqref{e:defbj}, the fact that $\theta$ is not in $\myB^\J$ gives information either on its support ``on the base'' or ``in the fibers''. We explore these two possibilities successively, and we seek for a contradiction in both cases.

\subparagraph{Case 1.} First, assume that there exist two points $x',y'\in S\Sigma_\J$ in the support of $\theta$ such that the projections $x$ and $y$ of $x'$ and $y'$ on $\mathbf{H}^\J$ are not in the same equivalence class for $\overset{\J}{\sim}$. Our goal is to write $\theta$ under the form
\begin{equation}\label{e:condontheta}
\theta=t\theta_1+(1-t)\theta_2, \qquad \text{with } 0<t<1, \quad \theta_1,\theta_2\in \D_{r}^\J, \quad \theta_1\neq \theta_2.
\end{equation}
By closedness of the equivalence classes $M_{\bullet}^\J$ (see definition in Section \ref{s:notationsthconverse}), there exists a small open set $U$ around $x$ such that any point in $\overline{U}$ does not belong to $M_{y}^\J$. Let $V$ be the union of all sets $M_{\bullet}^\J$ intersecting $U$, and $W =\mathbf{H}^\J\setminus V$. Denote by $V'\subset S^*\M$ (resp. $W'\subset S^*\M$) the set of points in $A_r^\J$ whose projection to $\mathbf{H}^\J$ belongs to $V$ (resp. $W$). Then \eqref{e:condontheta} holds with $\theta_1=\frac{\theta(\bullet\cap V')}{\theta(V')}$, $\theta_2=\frac{\theta(\bullet\cap W')}{\theta(W')}$ and $t=\theta(V')$. Thus  $\theta$ is not extremal, which is a contradiction.

\subparagraph{Case 2.} Assume now that there exist two points $(q,p),(q',p')\in S\Sigma_\J$ in the support of $\theta$ such that the homogeneous coordinates $[p_{z_1}:\cdots:p_{z_m}]$ and $[p_{z_1}':\cdots:p_{z_m}']$ are not equal. As in the previous case, taking $V'$ as the set of points $(q'',p'')$ in $S\Sigma_\J$ whose homogeneous coordinates $[p_{z_1}'':\cdots:p_{z_m}'']$ are close to $[p_{z_1}:\cdots:p_{z_m}]$ and $W'$ as the complementary set, we see that \eqref{e:condontheta} holds again with $\theta_1=\frac{\theta(\bullet\cap V')}{\theta(V')}$, $\theta_2=\frac{\theta(\bullet\cap W')}{\theta(W')}$ and $t=\theta(V')$. Thus  $\theta$ is not extremal, which is a contradiction. 

\subparagraph{Conclusion.} By case 1 we conclude that the pushforward $\pi_*\theta$ of $\theta$ through the canonical projection $\pi:S^*\M\rightarrow \M$ is the tensorial product of the Lebesgue measure in the components $\notin\J$ and of a measure supported on a single equivalence class $M_q^\J$. Thanks to the invariance of $\theta$ under the flows of $\partial_{z_j}$ for $j\in\J$, this last measure is the uniform Dirac delta measure (see Section \ref{s:notationsthconverse}) on $M_q^\J$. Using case 2 we see that there exist homogeneous coordinates $[p_{z_1}:\cdots:p_{z_m}]$ such that in the fibers of $S^*\M$, $\theta$ has mass only on points having these homogeneous coordinates. Thanks to the  invariance of $\theta$ under the flows of $\partial_{z_j}$ for $j\in\J$ we obtain that $\theta$ is of the form $\alpha_q\otimes \delta_p$, i.e., $\theta\in\myB^\J$.
\end{proof}

Thanks to the Krein--Milman theorem applied in the locally convex topological vector space consisting of all measures on $S^*\M$ endowed with the weak-$*_K$ convergence topology, it follows from Lemmas \ref{l:convcomp} and \ref{l:extremal} that
\begin{equation}\label{e:d1}
\text{$\myC^\J\cap \D_{r}^\J$ is dense in the weak-$*_K$ topology in $\D_{r}^\J$}.
\end{equation}
Now, we justify that
\begin{equation}\label{e:d2}
\text{any $\nu^\J\in\D^\J$ is the weak-$*_K$ limit of a sequence $\nu_r^\J\in\D_r^\J$ as $r\rightarrow +\infty$.}
\end{equation}
Let $\chi\in C^\infty(\R;[0,1])$ such that $\chi(x)=0$ for $x\leq 1$ and $\chi(x)=1$ for $x\geq2$. Let $\nu^\J\in\D^\J$. We set for $r\in\N^*$
\begin{equation*}
\nu_r^\J=c_r^\J\nu^\J\prod_{j\in\J}\chi\Big(\frac{|{p_{z_j}}|}{r}\Big)\in\D_r^\J.
\end{equation*}
Here $c_r^\J$ is a normalizing constant. Since $\mathbf{1}_{S\Sigma^\J}\nu^\J=\nu^\J$, it follows that $c_r^\J$ tends to $1$ as $r\rightarrow 0$.
By the dominated convergence theorem, we get that $\nu_r^\J$ converges in the \mbox{weak-*} topology towards $\nu^\J$, which proves \eqref{e:d2} since $C^K(S^*\M)\subset C^0(S^*\M)$.
Combining \eqref{e:metrictopology}, \eqref{e:d1} and \eqref{e:d2}, this concludes the proof of Fact \ref{f:density}.

\subsubsection{Proof of Fact \ref{l:simplebuildblock23}}

To prove Fact \ref{l:simplebuildblock23}, we start with a preliminary statement, concerning $\myB^\J$.
\begin{fact}\label{l:simplebuildblock1}
Any $\nu\in\myB^\J$ is a QL, associated to an $L^2(\M)$-normalized sequence of eigenfunctions which are also eigenfunctions of $R_j$ and $\Omega_j$ for $j\in\J$.
\end{fact}
\begin{proof}[Proof of Fact \ref{l:simplebuildblock1}]
Let $\nu=\mathscr{H}_{q,p}^\J\in\myB^\J$. In particular, $(q,p)\in S\Sigma_\J$. Thanks to Lemma \ref{l:step0}, we assume in the sequel that $q=0$. Without loss of generality, we assume furthermore that $\J=\{1,\ldots,J\}$ for some $1\leq J\leq m$.

We construct a sequence of eigenfunctions $(\varphi_k)_{k\in\N^*}$ of $-\Delta$ which admits $\mathscr{H}_{0,p}^\J$ as unique QL. In our construction, for any $k\in\N^*$, $\varphi_k$ belongs to the eigenspace $\mathcal{H}^\J_{(0),(\alpha_{j,k})}$ for some $(\alpha_{j,k})\in(\N^*)^\J$, and it does not depend on the variables in the $i$-th copy of $\mathbf{H}$ for $i\notin \J$. Note also that the $(0)$ appearing in $\mathcal{H}^\J_{(0),(\alpha_{j,k})}$ means that all $\varphi_k$ are eigenfunctions of $\Omega_j$ ($j\in\J$) with eigenvalue $2\times0+1=1$. Our goal is to choose adequately the $J$-tuples $(\alpha_{j,k})_{j\in\J}$. A similar argument for $m=1$ is done in the proof of Point 2 of Proposition 3.2 in \cite{de2018spectral}.

We fix a sequence of $J$-tuples $(\alpha_{1,k},\ldots,\alpha_{J,k})\in (\Z\setminus\{0\})^J$, for $k\in\N^*$, such that:
\begin{itemize}
\item For any $1\leq j\leq J$, $\alpha_{j,k}\rightarrow +\infty$ as $k\rightarrow +\infty$.
\item For any $1\leq j,j'\leq J$, 
\begin{equation} \label{convalphaversp}
\frac{\alpha_{j,k}}{\alpha_{j',k}}\underset{k\rightarrow+\infty}{\longrightarrow} \frac{p_{z_j}}{p_{z_{j'}}},
\end{equation}
where $[p_{z_1}:\cdots:p_{z_m}]$ are the homogeneous coordinates of $p$ in $S\Sigma$ (see Section \ref{s:flowsandprobabilities}).
\end{itemize}

Now, for any $k\in\N^*$, denoting by $\mathbf{1}$ the constant function equal to $1$ (on some copy of $\mathbf{H}$), we define 
\begin{equation} \label{e:defphik}
\varphi_k=\Phi_k^1\otimes \cdots\otimes \Phi_k^J\otimes \underbrace{\mathbf{1} \otimes\cdots\otimes \mathbf{1}}_{m-J \text{ times}},
\end{equation}
where, for $1\leq j\leq J$, 
\begin{equation}\label{e:Phiphi}
\Phi_k^j(x_j,y_j,z_j)=\phi_{j,k}(x_j,y_j)e^{i\alpha_{j,k}z_j}
\end{equation}
is an eigenfunction of $-\Delta_j$ (on the $j$-th copy of $\mathbf{H}$) with eigenvalue $|\alpha_{j,k}|$. The precise form of $\phi_{j,k}$ will be given below.

In the next paragraphs, we explain how to choose $\phi_{j,k}$ in order to ensure that $(\varphi_k)_{k\in\N^*}$ has a unique QL, which is $\mathscr{H}_{0,p}^\J$. 

We first follow some arguments of the proof of \cite[Proposition 3.2]{de2018spectral}. A Fourier expansion in the $z_j$ variable yields (see for example \cite[Section~2]{colin2calcul})
\begin{equation}\label{e:decompodeltabgamma}
-\Delta_j=\bigoplus_{\gamma\in\Z} B_\gamma, \  \ \ \ \text{where}  \ \ B_\gamma=A_\gamma^*A_\gamma+\gamma \ \text{for $\gamma\in\Z$}
\end{equation}
where the operators $A_\gamma, A_\gamma^*$ and $B_\gamma$ act on the space of functions of the form $e^{i\gamma z_j}g(x_j,y_j)$, with $A_\gamma=\partial_{x_j}+i\partial_{y_j}+\gamma x_j$ (the annihilation operator) and $A_\gamma^*=-\partial_{x_j}+i\partial_{y_j}+\gamma x_j$ (the creation operator). We have $[A_\gamma,A_\gamma^*]=2\gamma$, hence the eigenspace of $B_\gamma$ corresponding to the eigenvalue $|\gamma|$ is of the form $(\text{ker}(A_\gamma))e^{i\gamma z_j}$.

We note that the function 
\begin{equation*}
f_{j,k}(x_j,y_j)=\exp\Big(-\alpha_{j,k}\frac{x_j^2+y_j^2}{4}+\frac{i}{2}\alpha_{j,k}x_jy_j\Big)
\end{equation*} satisfies 
\begin{equation*}
\big(-\partial_{x_j}^2-(\partial_{y_j}-i\alpha_{j,k})^2\big)f_{j,k}=\alpha_{j,k} f_{j,k}
\end{equation*}
 on $\R^2$, and its mass concentrates as $k\rightarrow +\infty$ at the point $(x_j,y_j)=(0,0)$. Let $\chi:\R^2\rightarrow \R$ be a smooth cut-off function equal to $1$ near $0$ and with small support. Then 
\[ \chi(x_j,y_j)f_{j,k}(x_j,y_j) \] 
can be seen as a function on the $j$-th copy of $\mathbf{H}$. Up to multiplying $\chi f_{j,k}$ by a constant (depending on $j,k$) we can assume that its $L^2$-norm is equal to $1$. Then 
\[ B_{\alpha_{j,k}}\big(\chi f_{j,k}\big)=\alpha_{j,k}\chi f_{j,k}+o_{L^2}(1) \] 
since $B_{\alpha_{j,k}}=-\partial_{x_j}^2-(\partial_{y_j}-i\alpha_{j,k})^2$ locally. 

We denote by $\phi_{j,k}$ the projection of $\chi f_{j,k}$ on the $\alpha_{j,k}$-eigenspace of $B_{\alpha_{j,k}}$. Our goal is to prove that  
\begin{equation}\label{e:modequasimodeequal}
\phi_{j,k}=\chi f_{j,k}+o_{L^2}(1)
\end{equation}
as $k\rightarrow +\infty$. We can decompose
\begin{equation}\label{e:decompochif}
\chi f_{j,k}=\phi_{j,k}+r_{j,k}
\end{equation} 
with $r_{j,k}$ orthogonal to $\phi_{j,k}$. Applying $B_{\alpha_{j,k}}$ to \eqref{e:decompochif}, we obtain 
\begin{equation}\label{e:eqBalp}
B_{\alpha_{j,k}}r_{j,k}=\alpha_{j,k}r_{j,k}+o_{L^2}(1).
\end{equation}
We know that $B_{\alpha_{j,k}}$ has eigenvalues $(2n+1)\alpha_{j,k}, n\in\N$, hence its lowest eigenvalue $\alpha_{j,k}$ is well separated from the rest of the spectrum. This implies that 
\[ (B_{\alpha_{j,k}}r_{j,k},r_{j,k}) \geq (\alpha_{j,k}+\epsilon)\lVert r_{j,k}\rVert_{L^2}^2 \] 
for some $\epsilon>0$. Combined with \eqref{e:eqBalp}, we obtain that $r_{j,k}=o_{L^2}(1)$, which proves \eqref{e:modequasimodeequal}.

For the above choice of $\phi_{j,k}$, we consider $\varphi_k$ given by \eqref{e:defphik}. Setting
\[ f_k=\prod_{j=1}^J  \chi(x_j,y_j)f_{j,k}(x_j,y_j)e^{i\alpha_{j,k}z_j} \]
we deduce from \eqref{e:modequasimodeequal} that
\begin{equation}\label{e:quasimodemodeclose}
\varphi_k=f_k+o_{L^2}(1)
\end{equation}
 as $k\rightarrow +\infty$. 

Let $\nu$ be a QL of $(\varphi_k)_{k\in\N^*}$. We are going to prove that necessarily $\nu=\mathscr{H}_{0,p}^\J$. Firstly, 
\begin{align}
&\text{$\nu$ is supported where $x_j=y_j=0$}\text{ for any $j\in\J$;}\label{e:conclusion1}\\
&\qquad\qquad\text{$\nu$ is of the form $\nu=m^\J\otimes \ell^{\notin\J}$}\label{e:conclusion2}
  \end{align}
where $\ell^{\notin\J}$ has been introduced in Section \ref{s:mainresultssection} and $m^\J$ is a probability measure on $S^*\mathbf{H}^\J$.
  The first line comes from the fact that $|f_{j,k}|$ concentrates as $k\rightarrow +\infty$ on $x_j=y_j=0$.
The second line comes from the fact that $\varphi_k$ does not depend on $x_i,y_i,z_i$ for $i\notin \J$.

Also, let us justify that for any $n\in\N^*$,
\begin{equation} \label{e:varphikbienlocalise}
P_n^\J\varphi_k=\varphi_k
\end{equation}
 for $k$ sufficiently large (see \eqref{e:defPn} for the definition of $P_n^\J$, here $Z_i=\partial_{z_i}$). For this, we notice that 
 \begin{align*}
 \Big(\frac{\text{Id}-\Delta_{\sR}}{E}\Big)\varphi_k&=\overbrace{\Big(\frac{1+\sum_{j\in\J} |\alpha_{j,k}|}{1+\sum_{j\in\J} |\alpha_{j,k}|+|\alpha_{j,k}|^2}\Big)}^{:=\varepsilon_{k}}\varphi_k, \smallskip\\
\Big(\frac{Z_i^*Z_i}{E}\Big)\varphi_k&=0 \text{  when $i\notin \J$,}\smallskip\\
\Big(\frac{Z_j^*Z_j}{E}\Big)\varphi_k&= \underbrace{\Big(\frac{|\alpha_{j,k}|^2}{1+\sum_{j'\in\J} |\alpha_{j',k}|+|\alpha_{j',k}|^2}\Big)}_{:=\eta_{j,k}}\varphi_k \text{  when $j\in \J$}.
 \end{align*} 
Using \eqref{convalphaversp}, we see that $\varepsilon_{k}$ converges to $0$ as $k\rightarrow +\infty$, and $\eta_{j,k}$ converges for any $j\in\J$ to a non-zero limit as $k\rightarrow +\infty$. This is sufficient to deduce \eqref{e:varphikbienlocalise}.

From \eqref{e:varphikbienlocalise}, we conclude that the mass of any QL of $(\varphi_k)_{k\in\N^*}$ is contained in $S\Sigma_\J$ according to Lemma \ref{l:supportnuJemptyset} (see also Lemma \ref{l:propofPnJ}).  Applying then Lemma \ref{l:locmdm}, for any $1\leq i,j\leq J$, to the operator $\frac{\partial_{z_i}}{\partial_{z_j}}-\frac{p_i}{p_j}$, and using \eqref{convalphaversp}, we obtain that 
\begin{equation}\label{e:conclusion3}
\text{$\nu$ is supported in $\M\otimes \delta_p$.}
\end{equation}

Together with the fact that $|\varphi_k|^2$ does not depend on $z_1,\ldots,z_J$, this yields that
\begin{equation}\label{e:conclusion4}
\text{$\nu$ is invariant under $\partial_{z_1},\ldots,\partial_{z_J}$.}
\end{equation} 
Combining \eqref{e:conclusion1}, \eqref{e:conclusion2}, \eqref{e:conclusion3} and \eqref{e:conclusion4}, we obtain that $\nu=\mathscr{H}_{0,p}^\J$, which concludes the proof of Fact \ref{l:simplebuildblock1}. 
\end{proof}

\begin{proof}[Proof of Fact \ref{l:simplebuildblock23}]. 
We consider $\nu \in\myC^\J$, and we write
\[ \nu=\sum_{i\in\mathcal{F}}\beta_i\ \mathscr{H}_{q_i,p_i}^\J \]
where $\mathcal{F}$ is a finite set, $\sum_{i\in\mathcal{F}}\beta_i=1$, and for any $i\in\mathcal{F}$, $\beta_i\geq 0$, $(q_i,p_i)\in S\Sigma_\J$.
Note that if $i\neq i'$, either $\mathscr{H}_{q_i,p_i}^\J=\mathscr{H}_{q_i',p_i'}^\J$, or the supports (in $T^*\M$) of $\mathscr{H}_{q_i,p_i}^\J$ and $\mathscr{H}_{q_i',p_i'}^\J$ are disjoint. Therefore, possibly grouping terms in the above sum, we assume that the supports of $\mathscr{H}_{q_i,p_i}^\J$ and $\mathscr{H}_{q_i',p_i'}^\J$ are disjoint as soon as $i\neq i'$.

For $i\in\mathcal{F}$, using Fact \ref{l:simplebuildblock1}, we consider a sequence of eigenfunctions $(\varphi_k^i)_{k\in\N^*}$ with eigenvalues $(\lambda_k^i)_{k\in\N^*}$ and whose unique QL is $\mathscr{H}_{q_i,p_i}^\J$. According to Fact \ref{l:simplebuildblock1}, we can also assume that $\varphi_k^i \in\mathcal{H}^\J_{(0),(\alpha_{j,k}^i)}$ for some $J$-tuples $(\alpha_{j,k}^i)_{j\in\J}$. For the moment, the only condition imposed on the integers $\alpha_{j,k}^i$ is that they satisfy \eqref{convalphaversp}. For any fixed $i\in\mathcal{F}$ and any fixed $k\in\N^*$, we can multiply all the $\alpha_{j,k}^i$ by a common factor $c_{i,k}$, this does not change \eqref{convalphaversp}. Choosing these factors adequately, we can hence assume that
\begin{equation*}
\lambda_k^i:=\sum_{j\in\J}|\alpha_{j,k}^i|
\end{equation*}
does not depend on $i\in\mathcal{F}$ (but it depends on $k$). In other words,
\begin{itemize}
\item for any $1\leq j\leq J$, $\varphi_k^i$ is also an eigenvalue of $\Omega_j$ with eigenvalue $1$;
\item for any $i,i'\in\mathcal{F}$, $\lambda_k^i=\lambda_k^{i'}$ and we denote this common value by $\lambda_k$. This means that for any $i\in\mathcal{F}$, $\varphi_k^i$ belongs to the eigenspace of $-\Delta$ corrresponding to the eigenvalue $\lambda_k$.
\end{itemize}
Since $\mathscr{H}_{q_i,p_i}^\J$ and $\mathscr{H}_{q_i',p_i'}^\J$ have disjoint supports, it follows by Lemma \ref{l:joint} that the eigenfunction of $-\Delta$ with eigenvalue $\lambda_k$
\begin{equation*}
\varphi_k:=\sum_{i\in\mathcal{F}} \beta_i\varphi_k^i
\end{equation*} 
admits $\nu^\J$ as unique QL in the limit $k\rightarrow +\infty$.
\end{proof}

\subsubsection{Step 3}
Let us now finish the proof of Theorem \ref{t:conversereplacement}. Let 
\begin{equation*}
\nu_\infty=\sum_{\mathcal{J}\in\mathcal{P}\setminus \{\varnothing\}}c_\J\nu^\J
\end{equation*}
be a probability measure with $\nu^\J\in\D^\J$ and $c_\J\geq 0$ for any $\J\in\myP\setminus \{\varnothing\}$.

Let $(\varphi_k^\J)_{k\in\N^*}$ be a sequence of eigenfunctions of $-\Delta$ whose unique microlocal defect measure is $c_\J\nu^\J$. The fact that in the proof of Lemma \ref{l:buildblock1} we only impose the condition \eqref{convalphaversp} on the integers $\alpha_{j,k}$ guarantees that, for any $k\in\N^*$, one may choose all $\varphi_k^\J$, for $\J$ running over $\mathcal{P}\setminus \{\varnothing\}$, to have the same eigenvalue with respect to $-\Delta$. Therefore, 
\begin{equation*}
\varphi_k=\sum_{\J\in\mathcal{P}\setminus \{\varnothing\}} \varphi_k^\J
\end{equation*}
is also an eigenfunction of $-\Delta$. Moreover, since $S\Sigma_\J$ and $S\Sigma_{\J'}$ are disjoint for any distinct $\J,\J'\in \mathcal{P}\setminus \{\varnothing\}$, computing $(A\varphi_k,\varphi_k)$ for any $A\in\Psi^0(\M)$ in the limit $k\rightarrow +\infty$, we obtain by Lemma \ref{l:joint} that the unique QL of $(\varphi_k)_{k\in\N^*}$ is $\nu_\infty$. This concludes the proof of Theorem \ref{t:conversereplacement}. 

\section{Proof of Theorem \ref{t:goodconversereplacement}} \label{s:goodconverse}
In this Section, we finally prove Theorem \ref{t:goodconversereplacement}.  Precisely, we prove that there exists a QL $\nu$ such that the equation $\vec{\rho}_s^\J\nu=0$ is satisfied only for $\J=\{1,2\}$ and $s=(\frac12,\frac12)\in\mathbf{S}_{\{1,2\}}$.
For this, we keep the notation \eqref{e:Phiphi}. For $k\in\N^*$, we set 
\begin{equation*}
\varphi_k=C_k\big(e^{ikz_1}\phi_{1,k}(x_1,y_1)e^{ikz_2}\phi_{2,k}(x_2,y_2)+e^{i(k+1)z_1}\phi_{1,k}(x_1,y_1)e^{i(k-1)z_2}\phi_{2,k}(x_2,y_2)\big).
\end{equation*}
where $C_k>0$ is a normalizing constant, so that $\lVert \varphi_k\rVert_{L^2(\M)}=1$. Here $\phi_{1,k}$ and $\phi_{2,k}$ are eigenfunctions of $B_k$ (see \eqref{e:decompodeltabgamma}) with eigenvalue $k$. Theorem \ref{t:goodconversereplacement} will be a consequence of the following proposition:
\begin{proposition}\label{p:int}
$(\varphi_k)_{k\in\N^*}$ has a unique QL, which is
\begin{equation}\label{e:formulafornu}
\nu=C\big[(1+\cos(z_1-z_2))\delta_{x_1,x_2,y_1,y_2}\big]\otimes \delta_{p^0}
\end{equation}
where $C>0$ is a normalizing constant so that $\nu$ is a probability measure, $\delta_{x_1,x_2,y_1,y_2}$ stands for the Dirac mass on $x_1=x_2=y_1=y_2=0$, and the only non-null coordinates of $p^0=(p_{x_1}^0,p_{y_1}^0,p_{z_1}^0,\ldots,p_{x_m}^0,p_{y_m}^0,p_{z_m}^0)$ are $p^0_{z_1}=p^0_{z_2}\neq 0$.
\end{proposition}
\begin{proof}[Proof of Proposition \ref{p:int}]
Both 
\[ \mu_k^1=e^{ikz_1}\phi_{1,k}(x_1,y_1)e^{ikz_2}\phi_{2,k}(x_2,y_2) \] 
and 
\[ \mu_k^2=e^{i(k+1)z_1}\phi_{1,k}(x_1,y_1)e^{i(k-1)z_2}\phi_{2,k}(x_2,y_2) \] 
are eigenfunctions of $-\Delta$, hence $\varphi_k$ is an eigenfunction of $-\Delta$ (with associated eigenvalue $2k$). 
Let $\nu$ be a QL of $(\varphi_k)_{k\in\N^*}$. 

Firstly, we compute 
\[ |\varphi_k|^2=C_k^2|\phi_{1,k}(x_1,y_1)|^2|\phi_{2,k}(x_2,y_2)|^2(1+\cos(z_1-z_2)). \]
Denote by $\pi:S^*\M\rightarrow \M$ the canonical projection, and recall that $\pi_*\nu$ is a weak-* limit of the sequence of functions $|\varphi_k|^2$ on $\M$ (this follows by taking $A$ to run over the multiplication operators by continuous functions on $\M$ in Definition \ref{d:defmdm}). Using that the mass of $\phi_{1,k}$ (resp. $\phi_{2,k}$) concentrates at $x_1=y_1=0$ (resp. $x_2=y_2=0$) as justified in the proof of Fact \ref{l:simplebuildblock1}, we obtain that 
\begin{equation}\label{e:pinu}
\pi_*\nu=C\big[(1+\cos(z_1-z_2))\delta_{x_1,x_2,y_1,y_2}\big]
\end{equation}
 for some $C>0$.

We set $\J=\{1,2\}$. We notice that $Z_1^*Z_1=|\partial_{z_1}|^2$ and $Z_2^*Z_2=|\partial_{z_2}|^2$ act as multiplication by $|k|^2$ on $\mu_k^1$, and $Z_j^*Z_j=|\partial_{z_j}|^2$ acts as $0$ on $\mu_k^1$ for $j\notin\J$. Since $-\Delta$ acts as $2k$, we have for any fixed $n\in\N$
\[ P_n^\J\mu_1^k=\mu_1^k \]
when $k$ is sufficiently large (see definition of $P_n^\J$ in \eqref{e:defPn}). The same is true for $\mu_2^k$. Hence, $\nu$ gives no mass to $S^*\M\setminus S\Sigma_\J$ according to Lemma \ref{l:supportnuJemptyset}.

Applying Lemma \ref{l:locmdm} to the operator $\frac{\partial_{z_1}}{\partial_{z_2}}-1$ and the sequence of functions $\varphi_k$, we obtain that $\nu$ is supported where $p_{z_1}=p_{z_2}$. Due to \eqref{e:pinu}, this implies \eqref{e:formulafornu}.
\end{proof}
We now conclude the proof of Theorem \ref{t:goodconversereplacement}. We already know that $\nu$ is concentrated on $S\Sigma_\J$ for $\J=\{1,2\}$. We want now to know for which $s\in\mathbf{S}_\J$ there holds $\vec{\rho}_s^\J\nu=0$. Setting $s=(s_1,s_2,0,\ldots,0)$, according to \eqref{e:formulafornu} we have
\[ \vec{\rho}_s^\J\nu=0 \Leftrightarrow (s_1\partial_{z_1}+s_2\partial_{z_2})\cos(z_1-z_2)=0\Leftrightarrow s_1=s_2. \]
Hence $\nu$ is invariant under $\vec{\rho}_s^\J$ only for $s=(\frac12,\frac12)\in\mathbf{S}_{\{1,2\}}$.


\appendix

\section{Appendix}
\subsection{Classical pseudodifferential calculus} \label{a:pseudo}

We briefly gather some basic facts of pseudodifferential calculus used along this paper (see also \cite[Chapter XVIII]{hormander1985analysis}). 

Following our notations of Section \ref{s:intro}, we denote by $M$ a smooth compact manifold of dimension $n$. We write $\mathscr{S}_{\text{hom}}^k(M)$ for the set of positively homogeneous degree $k$ functions on the cone $T^*M\setminus \{0\}$., i.e., $a\in \mathscr{S}_{\text{hom}}^k(M)$ if $a\in C^\infty(T^*M)$ and there exists $R>0$ such that for any $(q,p)\in T^*M$ with $|p|\geq R$, and any $\lambda\geq 1$, we have $a(q,\lambda p)=\lambda^ka(q,p)$.  We also denote by $\mathscr{S}^k(M)$ the set of polyhomogeneous symbols of degree $k$. Hence, $a\in \mathscr{S}^k(M)$ if $a\in C^\infty(T^*M)$, and for any $j\in\N$ there exists $a_j\in \mathscr{S}^{k-j}_{\text{hom}}(M)$ such that for any $N\in \mathbb{N}$, $a-\sum_{j=0}^N a_j\in \mathscr{S}^{k-N-1}(M)$.

We denote by $\Psi^k(M)$ the space of classical (polyhomogeneous) pseudodifferential operators of order $k$ on $M$. The algebra $\Psi(M)$ of classical (polyhomogeneous) pseudodifferential operators on $M$ is graded according to the chain of inclusions $\Psi^{-\infty}(M)\subset\cdots\subset\Psi^{k}(M)\subset \Psi^{k+1}(M)\subset \cdots\,$.

To a pseudodifferential operator $A\in \Psi^m(M)$, we can associate its \emph{principal symbol} $\sigma_P(A)$, and the map $\sigma_P:\Psi^k(M)/\Psi^{k-1}(M)\rightarrow \mathscr{S}_{\text{hom}}^k(M)$ is bijective. A \emph{quantization} is a continuous linear mapping
\begin{equation*}
\Op:\mathscr{S}^0(M)\rightarrow \Psi^0(M)
\end{equation*}
with $\sigma_P(\Op(a))=a$. An example is obtained using partitions of unity and the standard quantization which is given in local coordinates by 
\begin{equation} \label{e:defstandardquantization}
\Op^{\text{st}}(a)f(q)=(2\pi)^{-n}\int_{\R^n\times\R^n} e^{i\langle q-q',p\rangle}a(q,p)f(q')dq'dp.
\end{equation}
This is the quantization we used by default in this paper. In $\M$, choosing local coordinates adapted to the product structure, we see that this quantization preserves the product structure: if $\J\subset\{1,\ldots,m\}$, and $a$ (resp. $a'$) depends only on the coordinates $(q_j,p_j)_{j\in\J}$ (resp. on the coordinates $(q_j,p_j)_{j\notin\J}$), then $[\Op^{\text{st}}(a),\Op^{\text{st}}(a')]=0$.

We have the following properties:
\begin{itemize}
\item If $A\in\Psi^k(M)$ and $B\in\Psi^{\ell}(M)$, then $AB\in \Psi^{k+\ell}(M)$ and $\sigma_P(AB)=\sigma_P(A)\sigma_P(B)$.
\item If $A\in \Psi^k(M)$ and $B\in \Psi^\ell(M)$, then $[A,B]\in \Psi^{k+\ell-1}(M)$ and 
\begin{equation*}
\sigma_P([A,B])=\frac1i\{\sigma_P(A),\sigma_P(B)\},
\end{equation*}
where the Poisson bracket is taken with respect to the canonical symplectic structure of $T^*M$.
\end{itemize}

\begin{lemma} \label{l:locmdm}
Let us assume that $\ell\in\N$ and $P\in \Psi^\ell(M)$ is elliptic in any cone contained in the complement of a closed conic set $F\subset T^*M$. Assume that $(u_k)_{k\in\N^*}$ is a bounded sequence in $L^2(M)$ weakly converging to $0$ and such that $Pu_k\rightarrow 0$ strongly in $L^2(M)$. Then any microlocal defect measure of $(u_k)_{k\in\N^*}$ is supported in $F$.
\end{lemma}
\begin{proof}
Let $\mu$ be a microlocal defect measure of $(u_k)_{k\in\N^*}$, i.e.,
\begin{equation*}
(\Op(a)u_{\sigma(k)},u_{\sigma(k)})\underset{k\rightarrow+\infty}{\longrightarrow} \int_{S^*M}a d\mu
\end{equation*}
for any $a\in\mathscr{S}^0(M)$, where $\sigma$ is an extraction. Let $a\in\mathscr{S}^0(M)$ be supported outside $F$. Let $Q\in\Psi^{-\ell}(M)$ be such that $PQ-I\in\Psi^{-1}(M)$ on the support of $a$. Then $Q\Op(a)P\in\Psi^0(M)$ has principal symbol $a$, and therefore
\begin{equation*}
(Q\Op(a)Pu_{\sigma(k)},u_{\sigma(k)})\underset{k\rightarrow+\infty}{\longrightarrow} \int_{S^*M}a d\mu.
\end{equation*}
Using that $Pu_{\sigma(k)}\rightarrow0$, we get $(Q\Op(a)Pu_{\sigma(k)},u_{\sigma(k)})\rightarrow0$ as $k\rightarrow +\infty$, and therefore $\int_{S^*M}a d\mu=0$. Hence, $\mu$ is supported in $F$.
\end{proof}

\subsection{The Martinet sub-Laplacian}  \label{a:supplementary} \label{a:Martinet}

In this section, we provide an example of a sub-Laplacian on a compact manifold which satisfies Assumption \ref{a:a} but which is not step 2, meaning that brackets of length $\geq 3$ of the $X_i$ are required to generate the whole tangent bundle, see \eqref{e:hormander}. 

To this end, we consider $M=(\R/2\pi\Z)^3$ with coordinates $x,y,z$, endowed with the Lebesgue measure $d\mu=dxdydz$. Let $A$ be a smooth $1$-form $A=A_xdx+A_ydy$, where $A_x$ and $A_y$ depend only on $x$ and $y$. The $2$-form $B=dA=(\partial_xA_y-\partial_yA_x)dx\wedge dy$ is the ``magnetic field'' and $b=\partial_xA_y-\partial_yA_x$ is its ``strength''. We consider the vector fields $X_1=\partial_x+A_x\partial_z$ and $X_2=\partial_y+A_y\partial_z$. Then, $[X_1,X_2]=b\partial_z$. Now, we choose $A$ so that $b$ vanishes along a closed curve in $(\R/2\pi\Z)^2_{x,y}$, and $(\partial_xb,\partial_yb)\neq0$ along this curve. This construction is classical, see \cite{montgomery95}. When adding the $z$-variable, this yields a surface $\mathscr{S}\subset M$, called Martinet surface, on which $[X_1,X_2]=0$ but some bracket of length $3$ of $X_1,X_2$ generates the missing direction of the tangent bundle thanks to $(\partial_xb,\partial_yb)\neq0$. In other words, the sub-Laplacian has step $3$ on $\mathscr{S}$. Nevertheless, Assumption \ref{a:a} is satisfied with $Z_1=\partial_z$.

\subsection{Quantum Limits of flat contact manifolds} \label{s:contact}
The study of QLs of higher dimensional contact manifolds is also an interesting problem. In this section, we prove that for the natural sub-Laplacian defined on the quotient of the Heisenberg group $\mathbf{H}_d$ of dimension $2d+1$ by one of its discrete cocompact subgroups, the invariance properties of QLs are much simpler than those described in Theorem \ref{t:mainQL}, even though ``frequencies'' show up: the part of the QL which lies in $S\Sigma$ is  invariant under the lift of the Reeb flow, as in the three-dimensional case. 

We first define the Heisenberg group in any odd dimension and the associated sub-Laplacian. For $d\geq 1$, we consider the group law on $\R^{2d+1}$ given by 
\begin{equation*}
(x,y,z) \star (x',y',z') =(x+x',y+y',z+z'-x\cdot y')
\end{equation*}
where $x,x',y,y'\in\R^d$ and $z,z'\in\R$. The Heisenberg group $\widetilde{\mathbf{H}}_d$ is the group $\widetilde{\mathbf{H}}_d=(\R^{2d+1},\star)$. We consider the subgroup $\Gamma_d=(\sqrt{2\pi\mathbb{Z}})^{2d}\times 2\pi\mathbb{Z}$ of $\widetilde{\mathbf{H}}_d$, and the left quotient $\mathbf{H}_d=\Gamma_d\backslash \widetilde{\mathbf{H}}_d$. We also define the $2d$ left invariant vector fields on $\mathbf{H}_d$ given by
\begin{equation*}
X_j=\partial_{x_j}, \qquad Y_j=\partial_{y_j}-x_j\partial_z
\end{equation*}
for $1\leq j\leq d$. We fix $\beta_1,\ldots,\beta_d>0$ satisfying $\prod_{j=1}^d\beta_j=1$, we set $\beta=(\beta_1,\ldots,\beta_d)$ and we consider the sub-Laplacian 
\begin{equation} \label{e:Deltabeta}
\Delta_{\beta}=\sum_{j=1}^d\beta_j(X_j^2+Y_j^2)
\end{equation}
which is an operator acting on functions on $\mathbf{H}_d$. The positive real numbers $\beta_j$ are sometimes called frequencies, see \cite{agrachev1996exponential}.

We set $\rho=h_Z|_{\Sigma}$, which is the Hamiltonian lift of the Reeb vector field $Z=\partial_z$ to $\Sigma$ (see \cite[Section 2.3]{de2018spectral} for properties of the Reeb vector field).
\begin{proposition} \label{p:contact}
Let $(\varphi_k)_{k\in\N^*}$ be a sequence of $L^2(\mathbf{H}_d)$ consisting of normalized eigenfunctions of $-\Delta_\beta$. Then, any QL $\nu_\infty$ associated to $(\varphi_k)_{k\in\N^*}$ and supported in $S\Sigma$ is invariant under $e^{t\vec\rho}$, the lift of the Reeb flow.
\end{proposition}
\begin{remark}
This result follows from \cite[Theorem 2.10(ii)(2)]{fermanian2020quantum}, but we provide here a simple self-contained proof which illustrates the averaging techniques used in Section \ref{s:prooftheo}.
\end{remark}
\begin{remark}
We do not expect such a result to be true when the frequencies $\beta_j$ are not constant on the manifold.
\end{remark}
\begin{proof}[Proof of Proposition \ref{p:contact}]
Denoting by $(q,p)$ the canonical coordinates in $T^*\mathbf{H}_d$, i.e., $q=(x_1,\ldots,x_d,y_1,\ldots,y_d,z)$ and $p=(p_{x_1},\ldots,p_{x_d},p_{y_1},\ldots,p_{y_d},p_{z})$, we know that
\begin{equation*}
\Sigma=\{(q,p)\in T^*\mathbf{H}_d, \ p_{x_j}=p_{y_j}-x_jp_{z}=0\}
\end{equation*} 
is isomorphic to $\mathbf{H}_d\times \R$.

Up to extraction of a subsequence, we may assume that $(\varphi_{k})_{k\in\N^*}$ has a unique QL $\nu_\infty$, which is supported in $S\Sigma$. We set $R=\sqrt{\partial_z^*\partial_z}$ and, on its eigenspaces corresponding to non-zero eigenvalues, we define $\Omega_j=-R^{-1}(X_j^2+Y_j^2)=-(X_j^2+Y_j^2)R^{-1}$ for $1\leq j\leq d$. On these eigenspaces, the sub-Laplacian acts as
\begin{equation*}
-\Delta_\beta = R\Omega=\Omega R\qquad \text{with} \quad \Omega=\sum_{j=1}^d\beta_j\Omega_j
\end{equation*}
and $[R,\Omega]=0$. 

Replacing $\varphi_k$ by $\chi(\frac{R^2}{\text{Id}-\Delta_\beta+R^2})\varphi_k$ for some smooth function $\chi\in C^\infty(\R)$ vanishing near $0$ and equal to $1$ near in a neighborhood of $1$ does not change the QL since $\nu_\infty$ is supported in $S\Sigma$. In the sequel, an operator $T$ is said microlocally supported in the $\chi$-neighborhood of $\Sigma$ if it verifies $T=\chi(\frac{R^2}{\text{Id}-\Delta_\beta+R^2})T\chi(\frac{R^2}{\text{Id}-\Delta_\beta+R^2})$.

If $B\in\Psi^0(\mathbf{H}_d)$ is microlocally supported in the $\chi$-neighborhood of $\Sigma$ and commutes with $\Omega$, then
\begin{align}
([B,R]\varphi_k,\varphi_k)&= \frac{1}{\lambda_k}(BR\varphi_k,-\Delta_\beta\varphi_k)-\frac{1}{\lambda_k}(RB(-\Delta_\beta)\varphi_k,\varphi_k) \nonumber \\
&= \frac{1}{\lambda_k}(BR\varphi_k, R\Omega\varphi_k)-\frac{1}{\lambda_k}(RBR\Omega\varphi_k,\varphi_k) \nonumber \\
&= \frac{1}{\lambda_k}([\Omega,RBR]\varphi_k,\varphi_k)\nonumber \\
&=0. \label{e:[B,R]=0}
\end{align}
Let $U(t)=U(t_1,\ldots,t_d)=e^{i(t_{1}\Omega_{1}+\cdots+t_d\Omega_d)}$ for $t=(t_1,\ldots,t_d)\in(\R/2\pi\Z)^d$. For $A\in \Psi^0(\mathbf{H}_d)$ microlocally supported in the $\chi$-neighborhood of $\Sigma$, we consider
\begin{equation*}
\widetilde{A}=\int_{(\R/2\pi\mathbb{Z})^d} U(-t)AU(t)dt
\end{equation*}
which is also microlocally supported in the $\chi$-neighborhood of $\Sigma$. We argue as in Section \ref{s:spec2}:  due to the definition of $\Sigma$, 
\[ \sigma_P(\Omega_j)=h_{R}^{-1}(h_{X_j}^2+h_{Y_j}^2) \] 
vanishes at order $2$ on $\Sigma$. Thus the Hamiltonian vector field associated to $\sigma_P(\Omega_j)$  vanishes on $\Sigma$, and the associated Hamiltonian flow is stationary on $\Sigma$.

Using Egorov's theorem as in the proof of Lemma \ref{l:propnuellN}, we deduce that $\sigma_P(A)$ and $\sigma_P(\widetilde{A})$ coincide on $\Sigma$. Moreover, as in the proof of Lemma \ref{l:propnuellN}, $[\widetilde{A},\Omega]=0$. Therefore, using the computation \eqref{e:[B,R]=0} with $B=\widetilde{A}$, we obtain
\begin{equation*}
\int_{\Sigma}\vec{\rho}(\sigma_P(A))d\nu_\infty=\int_{\Sigma}\vec{\rho}(\sigma_P(\widetilde{A}))d\nu_\infty=\lim_{k\rightarrow +\infty}\frac1i ([\widetilde{A},R]\varphi_k,\varphi_k)=0.
\end{equation*}
Combining the facts that it is true for any $A$ microlocally supported in the $\chi$-neighborhood of $\Sigma$ and that $\nu_\infty$ is supported in $S\Sigma$, this implies that $\nu_\infty$ is invariant under the flow $e^{t\vec\rho}$.
\end{proof}

\begin{ack}
I am very grateful to Emmanuel Tr\'elat and Yves Colin de Verdi\`ere, who taught me so much about the subject and answered countless questions I asked in Paris and in Grenoble, and to Suresh Eswarathasan, Luc Hillairet and Clotilde Fermanian Kammerer for numerous remarks and suggestions. I also thank Richard Lascar, Nicolas Lerner, Richard Montgomery and Gabriel Rivi\`ere for very interesting discussions, and the kind hospitality of Luigi Ambrosio and the Scuola Normale Superiore in Pisa, where part of this work was done. Finally, I am very grateful to an anonymous referee for his careful reading of the manuscript and his many suggestions.
\end{ack}

\begin{funding}
This work was partially supported by the grant ANR-15-CE40-0018 of the ANR (project SRGI). 
\end{funding}

\bibliographystyle{abbrv}

\end{document}